\documentclass[print]{elsarticle}
\pdfoutput=1
\usepackage{lineno,hyperref}
\usepackage{bm}
\usepackage{mathtools}
\usepackage{amssymb,amsthm}
\newtheorem{fact}{Fact}
\newtheorem{theorem}{Theorem}
\usepackage{import}
\usepackage{upgreek}
\usepackage{xcolor}
\usepackage{graphicx,subfigure}
\usepackage{enumitem}

\modulolinenumbers[5]

\journal{Computer Methods in Applied Mechanics and Engineering }

\begin{document}

\begin{frontmatter}

\title{Model order reduction of thermo-mechanical models with parametric convective boundary conditions: focus on machine tools}

\author[mymainaddress,mysecondaryaddress]{Pablo Hern\'andez-Becerro\corref{mycorrespondingauthor}}
\cortext[mycorrespondingauthor]{Corresponding author}
\ead{hernandez@inspire.ethz.ch}
\author[mymainaddress]{Daniel Spescha}

\author[mysecondaryaddress]{Konrad Wegener}

\address[mymainaddress]{inspire AG, Technoparkstrasse 1, 8005 Zurich, Switzerland}
\address[mysecondaryaddress]{Institute of Machine Tools and Manufacturing (IWF), ETH Zurich, Leonhardstrasse 21, 8092 Zurich,  Switzerland}

\begin{abstract}

This paper presents a parametric Model Order Reduction (MOR) method for weakly coupled thermo-mechanical Finite Element (FE) models of machine tools and other similar mechatronic systems. This work proposes a reduction method, Krylov Modal Subspace (KMS), and a theoretical bound of the reduction error. The developed method addresses the parametric dependency of the convective boundary conditions using the concept of system bilinearization. Additionally, this paper investigates the coupling between the reduced-order thermal system and the mechanical response. A numerical example shows that the reduced-order model captures the response of the original system in the frequency range of interest.

\end{abstract}


\end{frontmatter}


\section{Introduction}

Machine tools, such as milling machines, grinding machines, or lathes, are complex mechatronic systems and key components in the manufacturing process. Increasing the precision of machine tools enables to manufacture more accurate parts required in many engineering applications. Among the different error sources limiting the precision of machine tools, thermally induced deviations are the main contributor to geometric errors in manufactured parts, as stated by Mayr et al.\ \cite{2012_Mayr_thermal}. During the manufacturing process, the temperature distribution of machine tools is inhomogeneous and varies over time, leading to deformations of the structural parts. Internal and external sources are responsible for the time-varying temperature distribution, as summarized by Bryan \cite{Bry1990}. Internal sources refer to heat losses at the machine elements during the manufacturing process, such as friction at the bearings. External sources refer to the surroundings of the machine tool, such as fluctuations of the environmental temperature of the workshop.

Virtual prototypes are a great asset to improve the thermo-mechanical behavior of machine tools, as explained by Altintas et al.\ \cite{ALTINTAS2005115}. They enable to test different design alternatives virtually before a physical prototype is available. Virtual prototypes are numerical models that describe the physical processes leading to errors in manufactured parts. Thermo-mechanical models of machine tools are based on the finite element (FE) discretization of the heat transfer and elasticity equations. However, discretized thermo-mechanical models of machine tools have a large scale due their geometrical complexity, requiring a high computational effort. Many engineering applications use surrogate models, a computationally efficient model reproducing the physical behavior of the high-fidelity model. Among the different techniques to create surrogate models, this work focuses on projection based model order reduction (MOR). MOR projects the original model into a lower dimensional subspace containing the most relevant information of the response of the system. As stated by Benner et al.\ \cite{Benner2015}, the main benefit of MOR is that it retains the systems structure enabling the traceability of the dynamical evolution. Additionally, the use of the underlying system structure facilitates the derivation of theoretical error bounds of the surrogate model. In comparison to other surrogate modeling techniques, MOR is an intrusive method, i.e.\ it needs to access and modify the system matrices. 

Several works apply MOR to thermo-mechanical models of machine tools and other mechatronic systems. Herzog et al.\ \cite{Herzog2018} presented a thermo-mechanical model of a machine tool column. They concentrated on the problem of optimal sensor placement, i.e.\ selecting the optimal points to place temperature sensors in order to predict the thermally induced deviations. Finding the optimal location of the temperature sensors is an iterative process with a large amount of model evaluations. Thus, the authors proposed a reduced thermo-mechanical model based on proper orthogonal decomposition (POD). The thermo-mechanical model described the thermal response of a machine tool column to localized heat sources, representing the losses of a drive and friction of the ballscrew nut. Herzog et al.\ considered that the heat loads had a constant position over time, resulting in a time-invariant optimal sensor placement algorithm. However, relative movements between the machine tool axes lead to moving thermal contacts and heat loads. Therefore, efficient thermo-mechanical models of machine tools need to consider the position dependency of the thermal and mechanical response.

Lang et al.\ \cite{Lang2014} compared two methods to trace the position dependency of the thermal behavior of reduced thermal models of machine tools. In their first approach, Lang et al.\ divided the trajectory of the path into several segments and calculated for each of them a reduction basis by means of balanced truncation (BT). During the movement of the machine tool, the model switches to the local reduced system that is the closest to the actual position. The second reduction approach focused on a global reduction method, i.e.\ one projection matrix valid for all the positions of the axes. The proposed reduction method is a parametric iterative rational Krylov algorithm (IRKA), which for each parameter sample selects the optimal expansion points and tangential directions in a $\mathcal{H}_2$ optimal sense. Both reduction approaches showed similar accuracy compared to the full model, while the time required for reduction of the parametric IRKA was considerably smaller than for the local basis with BT. However, the amplitude of the thermal response of machine tools decays at higher excitation frequencies. Therefore, MOR approaches capturing the thermal response of the system in frequency range of interest are more efficient. Lang et al.\ considered the position dependency of the temperature response. However, the stiffness of the machine tool varies at different positions of the axes. Thus, it is required to couple efficiently the temperature distribution to a mechanical model that can reproduce the position-dependent mechanical response.    

Convective boundary conditions describe the interaction between the structural parts of the machine tool with the surrounding fluid media. Convective heat exchange is the main driving mechanism of the thermal response of the machine tools. It describes the heat exchange due to fluctuations of the environmental temperature \cite{Mian2013, Weng2018} or introduction of cooling fluid into the working space \cite{Shi2018}. The heat transfer coefficient (HTC) is the proportionality constant between the convective heat flux and the temperature difference between the fluid and structure. The modification of the conditions of the external fluid flow results in a variation of the HTC. Therefore, efficient thermo-mechanical models need to enable the modification of the HTC. In the context of projection-based MOR, the HTC is a model parameter that needs to be traced after reduction, i.e.\ parametric MOR techniques are required.

Benner et al.\ \cite{Benner2015} provided a comprehensive review of parametric MOR methods. The parametric reduction approaches are classified into two groups, namely local or global reduction bases. A local reduction basis is constructed evaluating the system response at several samples of the parameter space. In order to switch the value of the parameter, interpolating between reduced systems is required. There are several alternatives in the literature to perform this interpolation, such as interpolating between the locally reduced system matrices \cite{Panzer2010, Lee2017}, between the local subspaces \cite{Amsallem2008}, or between the transfer functions \cite{BaurBenner2011}. Parametric reduction with local bases provides efficient models valid for several discrete samples of the parameters. However, the convective boundary conditions of thermal models of machine tools vary continuously. Therefore, parametric MOR approaches with global reduction bases are more suitable for the models under investigation. These methods create a single set of projection bases for the values of the parameter space. 

Several reduction approaches with global bases rely on bilinearization of the systems equations. These methods express the parametric system in a bilinear form, which is a special type of non-linear systems. Phillips \cite{Phillips2003} used functional series expansion for the reduction of bilinear systems with rational interpolation. Continuing the work of Phillips, Bai and Skoogh \cite{BAI2006406} generalized the construction of the reduction basis matching a desired number of moments of the bilinear system. Breiten and Damm \cite{BREITEN2010443} proposed a reduction method for bilinear systems with Krylov subspace methods including expansion points at different values than zero. Benner and Breiten \cite{Benner2012} presented a reduction approach for bilinear systems based on IRKA, named BIRKA. The authors proved that BIRKA is optimal in $\mathcal{H}_2$ for parameter-varying systems. Bruns and Benner \cite{Bruns2015} applied BIRKA to create a heat transfer model of an electrical motor. They investigated the parameters concerning the thermal contacts and convective heat transfer.

This paper proposes a parametric MOR reduction approach for thermo-mechanical models of mechatronics systems, such as machine tools. The presented MOR method considers the characteristic behavior of these models in order to build the reduction basis, enabling the possibility to trace the parameters describing the convective boundary conditions. The main contributions of this work can be summarized as:

\begin{itemize}
	\item Developing a novel MOR approach which efficiently represents the thermal behavior of the system in the frequency range of interest. Associated to the MOR method, an a-priori error estimator is proposed in order to have an upper bound of the error in the frequency range of interest.
	\item Enabling the traceability of the parameters describing the convective boundary conditions after reduction.
	\item Coupling efficiently the thermal response to the mechanical structural deformations. For any temperature distribution, the reduced model needs to evaluate the thermally induced structural deformations. 

\end{itemize}

\section{Formulation of the parametric thermo-mechanical model}

Thermo-mechanical models of machine tools describe the temperature distribution and the associated structural deformations. The temperature distribution is a continuous function $T(t, \bm{z})$ defined at every point $\bm{z}$ of the domain $\Omega$ over time. Based on the energy conservation principle, the heat transfer equation is a partial differential equation (PDE) describing the temporal and spatial evolution of the temperature field as
\begin{equation}
	\rho c_p \dot{T}(t, \bm{z}) - \text{div}(\lambda T(t, \bm{z})) = 0 
	\label{eq:PDE}
\end{equation}
where $c_p$ is the specific heat capacity, $\rho$ is the material density, and $\lambda$ is the thermal conductivity. The definition of the boundary and initial conditions ensure a unique solution of the PDE. The Neumann boundary conditions can be interpreted as a heat flux applied to a surface $\Gamma_1$. Neumann boundary conditions are defined as
\begin{equation}
	\lambda \frac{\partial T(t, \bm{z})}{\partial n} = \dot{q}(t, \bm{z})
	\label{eq:Neummann}
\end{equation}
where $\dot{q}(t, \bm{z})$ is the heat flux applied on the boundary $\Gamma_1$. Convection is represented by Robin boundary conditions. The Robin boundary conditions are defined as 
\begin{equation}
	\lambda \frac{\partial T(t, \bm{z})}{\partial n} - h(t, \bm{z}) (T(t, \bm{z}) - T_{ext}(t, \bm{z})) = 0
	\label{eq:Robin}
\end{equation}
where $h(t, \bm{z})$ is the heat transfer coefficient and $T_{ext}(t, \bm{z})$ is the external temperature acting on $\Gamma_2$. The Robin boundary condition represents the convective heat exchange between an external fluid media and the structure or the thermal contact between two parts. Finally, the initial condition is 
\begin{equation}
	T(\bm{z}, 0) = T_0(\bm{z})
	\label{eq:Initial}
\end{equation}
over the whole domain $\Omega$.

The heat transfer equations describe the spatial distribution and temporal evolution of the  temperature. However, the thermally induced mechanical deviations are the main output of interest. Therefore, a coupled thermo-mechanical model is required. The thermo-mechanical coupling implies that the temperature field affects the structural deformation due to non-zero thermal expansion coefficients. However, the mechanical work resulting from the deformation of the structure does not alter the temperature field. Thus, a weakly coupled mechanical system is needed. The assumption of weak coupling holds for thermo-mechanical models with small strain rates. 

The force balance equation relates in continuum mechanics the stress tensor with the external forces as  
\begin{equation}
-\text{div}(\bm{\sigma}(t, \bm{z})) = \bm{f}
\label{eq:PDE_mech}
\end{equation}
where $\bm{\sigma}$ is the stress tensor at a point $\bm{z}$ in the domain $\Omega$ and $\bm{f}$ is the external force vector. Small deformations and strains are assumed, and thus a linear elastic model is considered. The stress tensor can be decomposed into an elastic part $\bm{\sigma}_e$ and a thermal part $\bm{\sigma}_{th}$. Considering the constitutive equation, they can be expressed as 
\begin{equation}
\bm{\sigma}_e = \frac{E}{1+\nu}\bm{\epsilon} + \frac{E \nu}{(1+\nu)(1-2\nu)}\text{trace}(\bm{\epsilon})\bm{I}
\end{equation} 
\begin{equation}
\bm{\sigma}_{th} = - \frac{E}{(1-2\nu)}\alpha(T-T_{ref})\bm{I}
\end{equation} 
where $\bm{\epsilon}$ is the strain tensor, $E$ is the Young modulus, $\nu$ is the Poisson modulus, $\bm{I}$ is the identity tensor, $\alpha$ is the thermal expansion coefficient, and $T_{ref}$ is the reference temperature. Under the assumption of small deformations, the strain tensor can be linearized as the gradient of the displacements $\bm{u}_s$ as
\begin{equation}
\bm{\epsilon} = \frac{1}{2}(\nabla \bm{u}_s + \nabla \bm{u}_s ^T)
\end{equation} 

Numerical methods enable the solution of heat transfer and elasticity equations for general complex domains. The FE method is a well-established numerical method to solve the heat transfer and elasticity equations. Bathe \cite{Bathe2006} provides a comprehensive review of FE methods. Several commercial and open-source FE simulation packages are available, being extensively used in computer aided engineering (CAE). The FE-discretization interpolates a scalar field inside an element $e$ with the values at the nodes. The interpolation functions are called antsatz functions $\bm{n}_e(\bm{z})$, defined for any point $\bm{z}$ inside the domain of the element $\Omega^e$. Considering the interpolation of the temperature field with the ansatz function and the principle of virtual work, the PDE of Equation \eqref{eq:PDE} to \eqref{eq:Initial} is transformed into a system of ordinary differential equations (ODE) for the element $e$. Assembling for all the elements of the FE-mesh, the ODE for the whole system can be obtained as   
\begin{equation}
	\bm{C}_{th} \bm{\dot{x}}(t) + \bm{K}_{cond}(t) \bm{x}(t) + \bm{K}_{conv}(t) \bm{x}(t) = \bm{q}_{ext}
	\label{eq:FEM_thermal}
\end{equation}
where $\bm{C}_{th}$ is the thermal capacity matrix, $\bm{K}_{cond}$ is the thermal conductivity matrix,  $\bm{K}_{conv}$ is the thermal convection matrix, $\bm{q}_{ext}$ is the thermal heat input vector, and $\bm{x}$ is the temperature state vector. 

From the different physical parameters describing Equation \eqref{eq:FEM_thermal}, the value of the HTC are of special interest for thermo-mechanical models of machine tools. The FE-discretization of the Robin boundary condition affects both the convection matrix $\bm{K}_{conv}$ and the heat input vector $\bm{q}_{ext}$. Equation \eqref{eq:Robin} defines the HTC, $h(t, \bm{z})$, as the proportionality constant between the convective heat flux and the temperature difference between the structure and an external temperature. For the following derivations, it is assumed that the spatial distribution of the HTC does not change over time. Therefore, the HTC can be expressed separating the spatial and temporal dependency as $h(t, \bm{z}) = h(t)w(\bm{z})$. Considering this, the parametric dependency of the convection matrix and heat input vector of an element $e$ can be expressed as 
\begin{equation}
	\bm{K}^e_{conv}(t) = h(t)\int_{\Gamma^e_2}^{} \bm{n}_e w(\bm{z} ) \bm{n}_e^T d \bm{z}
	\label{eq:convection_FEM}
\end{equation}
\begin{equation}
	\bm{q}^e_{ext} = h(t)\int_{\Gamma^e_2}^{} \bm{n}_e w(\bm{z}) T_{ext}(t, \bm{z})  d\bm{z}
	\label{eq:heat_FEM}
\end{equation} 

Thermal FE models usually discretize the convective boundary condition into a finite number of boundaries with the same HTC, 
\begin{equation}
	\Gamma_{2} = \bigcup\limits_{i=1}^{n_c}\Gamma_i
	\label{eq:discrete_boundaries}
\end{equation} 
where $n_c$ is the number of independent boundaries. The HTCs for all individual boundaries, $h_i(t)$, can be arranged in a parameter vector $\bm{p} \in \mathbb{R}^{n_c \times 1}$. Therefore, the system of ODE of Equation \eqref{eq:FEM_thermal} can be expressed in state space representation as a function of $\bm{p}$ as
\begin{equation}
	\bm{E} \dot{\bm{x}}(t) = \bm{A}(\bm{p}) \bm{x}(t) + \bm{B}\bm{u}(t)
	\label{eq:LTI}
\end{equation}
where $\bm{E} \in \mathbb{R}^{n \times n}$ is the left-hand side system matrix or mass matrix, $\bm{A}(\bm{p}) \in \mathbb{R}^{n \times n}$  is the system matrix, $\bm{B}\in \mathbb{R}^{n \times m}$ is the input matrix, $\bm{x}(t) \in \mathbb{R}^{n \times 1}$ is the state vector, and $\bm{u}(t)\in \mathbb{R}^{m \times 1}$ is the input vector.  Equation \eqref{eq:LTI} considers that the input vector also includes the corresponding value of the HTC. Considering the assumptions leading to Equation \eqref{eq:convection_FEM}, the parametric dependency of the state space representation of the Equation \eqref{eq:LTI} can be expressed in an affine form as
\begin{equation}
	\bm{E} \dot{\bm{x}}(t) = \bm{A} \bm{x}(t) +  \sum_{i=1}^{n_c}h_i(t)\bm{D}_i \bm{x}(t) + \bm{B}\bm{u}(t)
	\label{eq:system_distributed}
\end{equation}
where $\bm{D}_i$ represents the convection matrix for each boundary $\Gamma_i$. The temperature output vector of the system $\bm{y}_{th}(t)$ is defined as 
\begin{equation}
	\bm{y}_{th}(t) = \bm{C}_{therm}\bm{x}(t)
	\label{eq:LTI_output}
\end{equation}
where $\bm{C}_{therm}\in \mathbb{R}^{p \times n}$ is the thermal output matrix. The FE discretization of Equation \eqref{eq:PDE_mech} leads to system of linear equations
\begin{equation}
	\bm{K} \bm{x}_{mech} = \bm{K}_{th}(\bm{x}(t) - \bm{x}_{ref}) + \bm{f}_{ext}
	\label{eq:mechanical_system} 
\end{equation}
where $\bm{K}$ is the stiffness matrix, $\bm{K}_{th}$ is the thermal coupling matrix, $\bm{x}_{ref}$ is the reference temperature vector, $\bm{f}_{ext}$ is the external force vector, and $\bm{x}_{mech}$ is the displacement state vector. On one hand, the thermal states $\bm{x}$ are an input of the mechanical system. On the other hand, external mechanical forces, such as preloads at the machine tool elements, affect the response of the mechanical system. Therefore, instead of defining the mechanical deviations as an output of the system of Equation \eqref{eq:LTI}, a separate state $\bm{x}_{mech}$ is defined. Defining a separate state and system matrix has another practical advantage during model development. It enables the creation and validation of a mechanical model of the structure and then extend it in order to consider thermo-mechanical effects. Therefore, having an dedicated mechanical model is beneficial from a practical point of view. 

The state space representation of the mechanical system is 
\begin{equation}
	\bm{K}\bm{x}_{mech}(t) = \begin{bmatrix}
	\bm{K}_{th} & \bm{B}_{ext}
	\end{bmatrix}\begin{bmatrix}
	\bm{x}(t) \\ \bm{u}_{ext}(t)
	\end{bmatrix}
	\label{eq:coupling_state}
\end{equation} 
where $\bm{K}$ is the mechanical system matrix or stiffness matrix, $\bm{B}_{ext}$ is the input matrix of the external mechanical forces, and $\bm{u}_{ext}(t)$ is the input vector of the external mechanical forces.

In order to complete the state space representation, an output vector is defined as $\bm{y}_{mech}(t)$ 
\begin{equation}
	\bm{y}_{mech} (t) = \bm{C}_{mech} \bm{x}_{mech}
	\label{eq:LTI_output_mech}
\end{equation}
where $\bm{C}_{mech}$ is the output matrix. The output of the mechanical model are the displacements at several points of interest, such as the relative deviations between tool and workpiece. 

\section{Formulation of the reduced-order model}

The FE-discretization of thermo-mechanical models of machine tools leads to the system equations describing the thermal transient behavior. In order to increase the computational efficiency of the models, this work proposes to create surrogate models by projection-based MOR. MOR methods construct a subspace containing the most relevant information of the high fidelity model. 

Let  $\mathcal{V}$ be the reduced subspace of dimension $r \ll n$, where $n$ is the dimension of the original system. Let $\bm{V} \in \mathbb{R}^{n \times r}$, such that $\mathcal{V} = \text{span} (\bm{V})$. The reduced state of the system $\bm{\tilde{x}}(t)$ is defined as
\begin{equation}
	\bm{x}(t) \approx \bm{V} \bm{\tilde{x}}(t)
	\label{eq:reduced_state}
\end{equation}

The different MOR approaches propose different methods to construct the reduction basis $\bm{V}$. Once the reduction subspace is defined, the thermal system is projected into the reduction subspace substituting Equation \eqref{eq:reduced_state} into Equation \eqref{eq:system_distributed} and \eqref{eq:LTI_output}. The system can be expressed as
\begin{equation}
	\bm{E} \bm{V} \bm{\dot{\tilde{x}}}(t) = \bm{A} \bm{V} \bm{\tilde{x}}(t) +  \sum_{i=1}^{n_c}h_i(t)\bm{D}_i \bm{V} \bm{\tilde{x}}(t) + \bm{B}\bm{u}(t)
\end{equation}
\begin{equation*}
	\bm{\tilde{y}}(t) = \underbrace{\bm{C}_{therm}\bm{V}}_{\bm{\tilde{C}}_{therm}}\bm{\tilde{x}}(t)
\end{equation*}

By enforcing the Petrov-Galerkin condition, a basis $\bm{W}$ can be found such that
\begin{equation}
\underbrace{\bm{W}^T\bm{E} \bm{V}}_{\bm{\tilde{E}}}\bm{\dot{\tilde{x}}}(t) = 
\underbrace{\bm{W}^T\bm{A} \bm{V}}_{\bm{\tilde{A}}}\bm{\tilde{x}}(t) +
\sum_{i=1}^{n_c}h_i(t) \underbrace{\bm{W}^T\bm{D}_i \bm{V}}_{\bm{\bm{\tilde{D}}_i}}\bm{\tilde{x}}(t) +
\underbrace{\bm{W}^T\bm{B} }_{\bm{\tilde{B}}}\bm{u}(t)
\end{equation}
leading to the projected system matrices $\bm{\tilde{E}}$, $\bm{\tilde{A}}$, $\bm{\tilde{D}}_i$,  $\bm{\tilde{B}}$, and $\bm{\tilde{C}}_{therm}$ and the following reduced system
\begin{equation}
	\bm{\tilde{E}} \bm{\dot{\tilde{x}}}(t)  = \bm{\tilde{A}} \bm{\tilde{x}}(t) + \sum_{i=1}^{n_c}h_i(t)\bm{\tilde{D}}_i \bm{\tilde{x}}(t) + \bm{\tilde{B}} \bm{u}(t) 
\end{equation}
\begin{equation*}
	\bm{\tilde{y}}(t) = \bm{\tilde{C}}_{therm}\bm{\tilde{x}}(t)
\end{equation*}
Similar considerations can be made for the mechanical system of Equation \eqref{eq:coupling_state} and \eqref{eq:LTI_output_mech}.

This work considers one-sided projection, i.e.\ $\bm{W} = \bm{V}$. As explained by Antoulas \cite{Antoulas2005}, the main advantage of one-sided projection is that it ensures the preservation of the stability after projection.    

\section{Krylov Modal Subspace (KMS) reduction}

The goal of this work is to create a reduced system that reproduces the thermal response for any value of the parameters describing the convective heat exchange between the structure and the surrounding fluid. Equation \eqref{eq:system_distributed} presents an affine representation of the parametric dependency, separating the system matrix into two terms. The first step considers the thermal part of the system that does not depend on the parameter. A reduction basis is created such that the reduced system reproduces the response of the system where the values of the parameters are set to a fixed value. This section introduces the MOR approach and an error bound for the error between the original and the reduced-order model. In a second step, the projection basis is extended in order to enable the modification of the HTC after reduction, which is the focus of the next section. 

The projection basis needs to capture the most relevant part of the response of the system. The thermal models of mechatronic systems, such as machine tools, have a certain characteristic behavior compared to other general first order systems. Firstly, the amplitude of the thermal response of the system decays at higher excitation frequencies. Therefore, the reduced system only needs to approximate the thermal response of the system in a low frequency range. Secondly, the steady state response is relevant in order to describe the behavior of the system. Thus, the reduced system needs to match the steady state temperature field of the original system for the different input and output combinations. 

This work proposes a reduction method that considers the characteristic behavior of thermal models of machine tools to construct the projection basis. On one hand, moment matching Krylov subspace methods approximate the response of the system around an expansion point $s_e$. The expansion point is placed at a low frequency close to zero in order to match the steady state response of the system. On the other hand, the eigenvectors of the system up to a certain frequency are included in the projection basis. Therefore, the thermal reduced system reproduces the thermal behavior of the system in the frequency range of interest. Spescha \cite{Spescha_Diss} introduced the concept of KMS reduction with application in structural dynamics. This paper extends the KMS method to first order systems and presents a theoretical bound of the reduction error.

Let $\mathcal{V}_k \subset \mathbb{R}^{n} $ be the Krylov subspace with one expansion point $s_e$, such that 

\begin{equation}
	\mathcal{V}_k = \text{span}((s_e\bm{E-\bm{A}})^{-1}\bm{B})
	\label{eq:subspace_krylov}
\end{equation}

An orthonormal basis $\bm{V}_k$ of $\mathcal{V}_k$  can be constructed, such that $\mathcal{V}_k = \text{span} (\bm{V}_k)$. The original system of Equation \eqref{eq:LTI} can be projected into the subspace $\mathcal{V}_k$. The reduced system matches the response of the system around the expansion point $s_e$. If an expansion point $s_e$ is chosen close to 0, the reduced system matches the steady state response. 

Let $\mathcal{V}_{\mu} \subset \mathbb{R}^{n} $ be the truncated modal subspace. The subspace $\mathcal{V}_{\mu}$ is the span of the first $\mu$ eigenvectors of the system of Equation \eqref{eq:LTI}, defined as 
\begin{equation}
	\mathcal{V}_{\mu} = \text{span}(\begin{bmatrix}
	\bm{\phi}_1  & \bm{\phi}_2& \dots& \bm{\phi}_{\mu}
	\end{bmatrix})
	\label{eq:subspace_modal}
\end{equation}
where $\bm{\phi}_i$ is an eigenvector associated to the eigenvalue $\alpha_i$ such that $(s_e\bm{E} - \bm{A})\bm{\phi}_i = \alpha_i \bm{\phi}_i$. The system matrices $\bm{A}$ and $\bm{E}$, defined in Equation \eqref{eq:FEM_thermal}, \eqref{eq:convection_FEM}, \eqref{eq:heat_FEM}, and \eqref{eq:LTI}, are negative semi-definite, resulting in all real non-positive eigenvalues. Therefore, the eigenvectors form an orthonormal basis $\bm{V}_{\mu}$ of $\mathcal{V}_{\mu}$, such that $\mathcal{V}_{\mu} = \text{span} (\bm{V}_{\mu})$. The original system of Equation \eqref{eq:LTI} can be projected into the subspace $\mathcal{V}_{\mu}$. The reduced system approximates the dynamic thermal behavior in a certain frequency range. 

The KMS reduction projects the system of Equation \eqref{eq:LTI} by means of a projection matrix $\bm{V}$. The basis $\bm{V}$ spans linear subspace $\mathcal{V}  \subset \mathbb{R}^{n} $, i.e.\ $\mathcal{V} = \text{span} (\bm{V})$, such that

\begin{equation}
	\mathcal{V} = \mathcal{V}_{\mu} + \mathcal{V}_k = \{\bm{x} \in  \mathbb{R}^{n} \; / \; \exists \bm{v}_1 \in \mathcal{V}_{\mu} \; \bm{v}_2 \in \mathcal{V}_{k} \; \bm{x} = \bm{v}_1 +\bm{v}_2 \}
	\label{eq:subspace_KMS}
\end{equation}

The subspace $\mathcal{V}$ captures the information about both the steady state response and the thermal transient behavior of the system. Therefore, this projection basis satisfies the requirement for accurate approximation of the thermal behavior of mechatronic systems.

\subsection{A-priori error estimator}

The reduced system needs to represent accurately the thermal response of the system in the frequency range of interest, i.e.\ for all $\omega \in [0 , \omega_{max}]$. The upper bound of the frequency range of interest, $\omega_{max}$, determines how many eigenvectors, $\mu$, are to be included in the KMS reduction basis. Therefore, this work proposes an a-priori error estimator for the KMS method that relates the maximum eigenfrequency, $\omega_{\mu} = \left\lvert  \alpha_{\mu}\right\rvert$, of the KMS basis with the frequency range of interest. 

The error estimator presented in this section considers that the input and output matrix of Equation \eqref{eq:LTI} and \eqref{eq:LTI_output} are the same, i.e.\ $\bm{B} = \bm{C}^T_{therm}$, leading to one-sided projection. For the derivation of the error estimator of the KMS method, the Krylov subspace defined in Equation \eqref{eq:subspace_krylov} has an expansion point $s_e \in \mathbb{R}$ close to zero and a single iteration. Before introducing the error estimator, some definitions and preliminary results are introduced. 

Firstly, a suitable error definition is required. Different error values are proposed in the literature, as summarized by Benner et al.\ \cite{Benner2015}. Let $\bm{E}(j\omega)$ be the absolute reduction error frequency response function (FRF) defined for each frequency $\omega$ as
\begin{equation}
\bm{E}(j\omega) = \bm{H}(j\omega) - \bm{\tilde{H}}(j\omega)\
\end{equation}  
where $\bm{H}(j\omega)$ and $\bm{\tilde{H}}(j\omega)$ are the FRF of the original and reduced system respectively. $\bm{E}(j\omega)$ is a matrix of dimension $p$ (number of outputs) by $m$ (number of inputs). Let $e_{ij}(j\omega)$ be the relative error for the $i$th input and $j$th output combination as 
\begin{equation}
e_{ij}(j\omega) = \frac{ h_{ij}(j\omega) - \tilde{h}_{ij}(j\omega)  }{ h_{ij}(j\omega) }
\label{eq:error_definition_ij}
\end{equation}
where $h_{ij}(j\omega)$ is the element in $i$th row and $j$th column of $\bm{H}(j\omega)$, and $\tilde{h}_{ij}(j\omega)$ is the element in $i$th row and $j$th column of $\bm{\tilde{H}}(j\omega)$. 

Secondly, some remarks about the subspaces associated to the KMS $\mathcal{V}$ are required. The subspace $\mathcal{V}_{\nu}\subset \mathbb{R}^{n}$ can be defined as the subspace of the remaining modes not included in $\mathcal{V}_{\mu}$, i.e.\
\begin{equation}
\mathcal{V}_{\nu} = \text{span}(\begin{bmatrix}
\bm{\phi}_{\mu+1}  & \bm{\phi}_{\mu+2}& \dots& \bm{\phi}_{n}
\end{bmatrix})
\label{eq:subspace_remaining}
\end{equation}

Due to the properties of the system matrix, the subspace $\mathcal{V}_{\mu}$ is the orthogonal complement of $\mathcal{V}_{\nu}$, such that $\mathbb{R}^{n} = \mathcal{V}_{\mu} \oplus \mathcal{V}_{\nu}$. Let $\bm{\Phi}$ be a matrix whose columns are the eigenvectors $\bm{\phi}_i$ of the system of Equation \eqref{eq:LTI}, normalized to the capacity matrix such that $\bm{\Phi}^T\bm{E} \bm{\Phi} = \bm{I}$. The system of Equation \eqref{eq:LTI} and \eqref{eq:LTI_output} can be expressed in modal coordinates $ \bm{x} = \bm{\Phi}\bm{x}_m$ as
\begin{equation}
\bm{I} \dot{\bm{x}}_m(t) = \bm{\Omega} \bm{x}_m(t) + \bm{\Phi}^T\bm{B}\bm{u}(t) =  \bm{\Omega} \bm{x}_m(t) + \bm{B}_m\bm{u}(t)
\label{eq:LTI_modal}
\end{equation}
\begin{equation}
\bm{y}_{th}(t) = \bm{C}_{therm}\bm{\Phi}\bm{x}_m(t) = \bm{C}_m\bm{x}_m(t)
\label{eq:LTI_modal_output}
\end{equation}
where $\bm{\Omega} = \text{diag}(\alpha_1, \dots, \alpha_n)$ is a diagonal matrix with the $n$ eigenvalues $\alpha_k$ of the system. The system matrix can be expressed as block matrix, such that
\begin{equation}
\begin{bmatrix}
\bm{I}_{\mu} & 0  \\
0           & \bm{I}_{\nu}
\end{bmatrix}\begin{bmatrix}
\dot{\bm{x}}_{\mu} \\
\dot{\bm{x}}_{\nu}
\end{bmatrix} = \begin{bmatrix}
\bm{\Omega}_{\mu} & 0  \\
0           & \bm{\Omega}_{\nu}
\end{bmatrix}\begin{bmatrix}
\bm{x}_{\mu} \\
\bm{x}_{\nu}
\end{bmatrix}+\begin{bmatrix}
\bm{B}_{\mu} \\
\bm{B}_{\nu}
\end{bmatrix}\bm{u}(t)
\label{eq:LTI_block_modal}
\end{equation}
\begin{equation}
\bm{y}_{th} = \begin{bmatrix}
\bm{C}_{\mu} &&
\bm{C}_{\nu}
\end{bmatrix}\begin{bmatrix}
\bm{x}_{\mu} \\
\bm{x}_{\nu}
\end{bmatrix}
\label{eq:LTI_block_modal_output}
\end{equation}

Let $\mathcal{V}_{\nu k} \subset \mathbb{R}^{n} $ be the Krylov subspace with one moment around the expansion point $s_e \in \mathbb{R}$ of the original system projected into the subspace $\mathcal{V}_{\nu}$, such that 

\begin{equation}
\mathcal{V}_{\nu k} = \text{span}((s_e\bm{I_{\nu}-\bm{\Omega_{\nu}}})^{-1}\bm{B_{\nu}})
\label{eq:subspace_rkrylov}
\end{equation}

The definition of $\mathcal{V}_{\nu k}$ leads to a preliminary result expressed in Fact \ref{fact:r_modal}. 

\begin{fact}
	
	The KMS $\mathcal{V}$ is the direct sum of $\mathcal{V}_{\mu}$ and $\mathcal{V}_{\nu k}$.	
	\begin{equation}
	\mathcal{V} = \mathcal{V}_{\mu} \oplus  \mathcal{V}_{\nu k}
	\end{equation}
	\label{fact:r_modal}
\end{fact}

\begin{proof}
	
	The Krylov subspace of the system in modal coordinates Equation \eqref{eq:LTI_modal} is 
	
	\begin{equation}
	\mathcal{V}_k = \text{span}((s_e\bm{I}-\bm{\Omega})^{-1}\bm{B}_m)
	\end{equation}
	
	As $s_e\bm{I}+\bm{\Omega}$ is a diagonal matrix, the Krylov subspace of the original system can be expressed as 
	
	\begin{equation}
	\mathcal{V}_k = \mathcal{V}_{\mu k} + \mathcal{V}_{\nu k} = \text{span}((s_e\bm{I}_{\mu}-\bm{\Omega}_{\mu})^{-1}\bm{B}_{\mu}) + \text{span}((s_e\bm{I}_{\nu}-\bm{\Omega}_{\nu})^{-1}\bm{B}_{\nu})
	\end{equation}
	
	The $\mathcal{V}_{\mu k}$ is a subspace of $\mathcal{V}_{\mu}$, i.e.\ $\mathcal{V}_{\mu k} \subset \mathcal{V}_{\mu}$. Therefore, $\mathcal{V} = \mathcal{V}_k +\mathcal{V}_{\mu} = \mathcal{V}_{\mu k} + \mathcal{V}_{\nu k} + \mathcal{V}_{\mu} = \mathcal{V}_{\nu k} + \mathcal{V}_{\mu} $. 
	
	In addition, $\mathcal{V}_{\nu k}$ is also a subspace of $\mathcal{V}_{\nu}$, i.e.\ $\mathcal{V}_{\nu k} \subset \mathcal{V}_{\nu}$. Since $\mathcal{V}_{\nu} \cap \mathcal{V}_{\mu} = \bm{0}$, it follows that $\mathcal{V}_{\nu k} \cap \mathcal{V}_{\mu} = \bm{0}$. This implies that $\mathcal{V}_{\nu k}$ is the orthogonal complement of $\mathcal{V}_{\mu}$, i.e.\ $\mathcal{V} =  \mathcal{V}_{\nu k} \oplus \mathcal{V}_{\mu}$

\end{proof}

Fact \ref{fact:r_modal} relates the subspaces $\mathcal{V}$, $\mathcal{V}_{\mu}$, and $\mathcal{V}_{\nu k}$, stating that any vector in $\mathcal{V}$ can be decomposed uniquely into two terms, one in $\mathcal{V}_{\mu}$ and another one in $\mathcal{V}_{\nu k}$. This property is useful to separate the error of the reduced system, as shown in Fact \ref{fact:separation_error}.   

\begin{fact}
	
	The error $e_{ij}(j \omega)$ of Equation \eqref{eq:error_definition_ij} can be separated as the product of two terms	
	\begin{equation}
	e_{ij}(j \omega) = -e_{\mu_{ij}}(j\omega )e_{\nu k_{ij}}(j\omega )
	\label{eq:error_separation}
	\end{equation}
	being $e_{\mu_{ij}}(j\omega )$ and $e_{\nu k_{ij}}(j\omega )$ defined as 
	\begin{equation}
	e_{\mu_{ij}}(j \omega) = \frac{h_{ij}(j\omega ) - \tilde{h}_{\mu_{ij}}(j\omega )}{h_{ij}(j \omega )}
	\label{eq:modal_error}
	\end{equation}
	\begin{equation}
	e_{\nu k_{ij}}(j \omega ) = \frac{\tilde{h}_{\nu k_{ij}}(j \omega ) - \tilde{h}_{\nu_{ij}}(j \omega )}{\tilde{h}_{\nu_{ij}}(j \omega )}
	\label{eq:error_reduced_krylov}
	\end{equation}
	where $\tilde{h}_{\mu_{ij}}(j\omega )$ is the element in $i$th row and $j$th column of the FRF of the system projected into $\mathcal{V}_{\mu}$, $\tilde{h}_{\nu_{ij}}(j\omega )$ is the element in $i$th row and $j$th column of the FRF of the system projected into $\mathcal{V}_{\nu}$, and $\tilde{h}_{\nu k_{ij}}(j\omega )$ is the element in $i$th row and $j$th column of the FRF of the system projected into $\mathcal{V}_{\nu k}$.
	
	\label{fact:separation_error}
\end{fact}

\begin{proof}
	The Fact \ref{fact:r_modal} enables to express the FRF of the reduced system as $\tilde{h}_{ij}(j \omega)  = \tilde{h}_{\mu_{ij}}(j \omega) + \tilde{h}_{\nu k_{ij}}(j \omega)$. Additionally, the transfer function of the original system can be expressed as $h_{ij}(j \omega)  = \tilde{h}_{\mu_{ij}}(j \omega) + \tilde{h}_{\nu_{ij}}(j \omega)$. Substituting in the error definition 
	\begin{equation}
	e_{ij}(j \omega) = \frac{\tilde{h}_{\mu_{ij}}(j \omega) + \tilde{h}_{\nu_{ij}}(j \omega) - \tilde{h}_{\mu_{ij}}(j \omega) - \tilde{h}_{\nu k_{ij}}(j \omega)}{h_{ij}(j \omega)}
	\end{equation}
	
	Multiplying the previous expression by $\frac{\tilde{h}_{\nu_{ij}}(j \omega)}{\tilde{h}_{\nu_{ij}}(j \omega)}$, the following is obtained
	\begin{equation}
	e_{ij}(j \omega) = \frac{\tilde{h}_{\nu_{ij}}(j \omega)  - \tilde{h}_{\nu k_{ij}}(j \omega)}{h_{ij}(j \omega)}\frac{\tilde{h}_{\nu_{ij}}(j \omega)}{\tilde{h}_{\nu_{ij}}(j \omega)}
	\end{equation}
	
	Substituting $\tilde{h}_{\nu_{ij}}(j \omega)$ in the numerator by $h_{ij}(j \omega)  - \tilde{h}_{\mu_{ij}}(j \omega)$ and reorganizing the terms, the error separation of Equation \eqref{eq:error_separation} is obtained as
	\begin{equation}
	e_{ij}(j \omega) = -\frac{\tilde{h}_{\nu k_{ij}}(j \omega) - \tilde{h}_{\nu_{ij}}(j \omega)}{\tilde{h}_{\nu_{ij}}(j \omega)}\cdot\frac{h_{ij}(j \omega)  - \tilde{h}_{\mu_{ij}}(j \omega)}{h_{ij}(j \omega)}
	\end{equation}

\end{proof}

The result of Fact \ref{fact:separation_error} enables the separation of the error in two terms. The goal is to find an bound for the error $e_{ij}(j \omega)$. The first step for an error bound of the KMS method is to determine an upper bound for the term $e_{\nu k_{ij}}(j \omega)$. The following theorem shows a theoretical error bound for all frequencies of the term $e_{\nu k_{ij}}(j \omega)$, provided that the input and the output are the same, namely $i = j$. 

\begin{theorem}
	The magnitude of the error $e_{\nu k_{ii}}(j \omega)$ defined in Equation \eqref{eq:error_separation} is bounded by $e_{est}(j\omega)$ defined as
	\begin{equation}
	\left \rvert e_{est}(j\omega) \right \rvert = \frac{\omega^2 + s_e^2}{\omega ^2 + w_m^2} > \left \rvert e_{rk_{ii}}(j\omega) \right \rvert
	\label{eq:error_bound}
	\end{equation}
	for all $\omega \in [0 , \infty]$ such that $\omega \in \mathbb{R}$ given that $\omega_{m} < \omega_{mo+1}$ and that the input and the output are the same, i.e.\ $i = j$. 
	
	\label{theorem:error}
\end{theorem}

\begin{proof}
	The FRF of the system projected into $\mathcal{V}_{\nu}$, is
	\begin{equation}
	\tilde{h}_{\nu_{ii}}(j\omega) = \sum_{k=1}^{\nu}\frac{b_k^2}{(j\omega + \omega_{k})}
	\label{eq:remaining_transfer}
	\end{equation}
	where $\nu = n - \mu$ is the dimension of the subspace $\mathcal{V}_\nu$, $\omega_{k}$ are the absolute value of the eigenvalues of the system, and $b_k$ corresponds to the $k$th element of the $i$th column of $\bm{B}_{\nu}$. The subspace $\mathcal{V}_{\nu k}$ can be defined according to Equation \eqref{eq:subspace_rkrylov} as $\text{range}(\bm{v})$ where	
	\begin{equation}
	\bm{v}^T = \begin{bmatrix}
	\frac{b_{1}}{\omega_{1}+s_e}  & 
	\frac{b_{2}}{\omega_{2}+s_e}& 
	\dots &
	\frac{b_{k}}{\omega_{k}+s_e}& 
	\dots &
	\frac{b_{r}}{\omega_{r}+s_e}
	\end{bmatrix}
	\end{equation}
	Projecting the system of Equation \eqref{eq:remaining_transfer} into $\mathcal{V}_{\nu k}$, the following system is obtained
	
	\begin{equation}
	\tilde{e} \dot{\tilde{x}} + \tilde{a}\tilde{x} =  \tilde{b} u(t) 
	\end{equation}
	where $\tilde{e}$, $\tilde{a}$, and $\tilde{b}$ take the following values 
	
	\begin{equation}
	\tilde{a} = \sum_{k=1}^{\nu} \frac{b_{k}^2}{(\omega_{k}+s_e)^2}\omega_{k}
	\label{eq:FRF_kr}
	\end{equation}
	
	\begin{equation*}
	\tilde{e} = \sum_{k=1}^{\nu} \frac{b_{k}^2}{(\omega_{k}+s_e)^2}
	\end{equation*}
	
	\begin{equation*}
	\tilde{b} = \sum_{k=1}^{\nu} \frac{b_{k}^2}{\omega_{k}+s_e} = \tilde{c}
	\end{equation*}
	
	The transfer function of the reduced system $\tilde{h}_{\nu k}$ is
	\begin{equation}
	\tilde{h}_{\nu k_{ii}}(j \omega) = \frac{\tilde{b}^2}{j \omega \tilde{e} +\tilde{a}} =\frac{\left(\sum_{k=1}^{\nu} \frac{b_{k}^2}{\omega_{k}+s_e}\right)^2}{j \omega \sum_{k=1}^{\nu} \frac{b_{k}^2}{(\omega_{k}+s_e)^2}+ \sum_{k=1}^{\nu} \frac{b_{k}^2}{(\omega_{k}+s_e)^2}\omega_{k} } 
	\end{equation}
	
	Firstly, some properties of the error $e_{\nu k_{ii}}(j \omega)$ need to be discussed. According to Equation \eqref{eq:error_reduced_krylov}, the poles of the transfer function of the real error, $e_{\nu k_{ii}}(s)$ with $s \in \mathbb{C}$, are the poles of $\tilde{h}_{\nu k_{ij}}(s)$, the poles $\tilde{h}_{\nu_{ii}}(s)$, and the zeros of $\tilde{h}_{\nu_{ii}}(s)$. Given that the system matrices are positive semi-definite, all the poles of $\tilde{h}_{\nu_{ii}}(s)$ are real. Additionally, the poles of $\tilde{h}_{\nu_{ii}}(s)$ are real, as it can be seen from Equation \eqref{eq:FRF_kr}. Furthermore, the zeros of the transfer function $\tilde{h}_{\nu_{ii}}(s)$ are real numbers. This fact can be shown by contradiction. Let $s = \alpha \pm j\beta$ be zeros of the transfer function $\tilde{h}_{\nu_{ii}}(s)$, such that $\beta > 0 $. The zero of the transfer function needs to satisfy that 
	
	\begin{equation}
		\tilde{h}_{\nu_{ii}}(\alpha + j\beta) = \sum_{k=1}^{\nu}\frac{b_k^2}{(\alpha + j\beta + \omega_{k})} = \sum_{k=1}^{\nu}\frac{b_k^2 }{((\alpha +\omega_{k})^2 + \beta^2)} (\alpha +  \omega_{k} - j\beta )  = 0
	\end{equation}
	
	Given that $\beta > 0 $ and $\frac{b_k^2 }{((\alpha +\omega_{k})^2 + \beta^2)} \geq 0$, the transfer function is only zero for $\beta = 0$, reaching a contradiction. Therefore, the zeros of transfer function $\tilde{h}_{\nu_{ii}}(s)$ are real. Thus, the poles of $e_{\nu k_{ii}}(j \omega)$ are all real. This leads to a smooth FRF response, without any resonance frequency. This property is relevant for the derivation of the error bound.
	
	In order to proof that the proposed estimator of Equation \eqref{eq:error_bound} bounds the error for all frequencies, several counterexamples are created, which are illustrated in Figure \ref{fig:couter}. Figure \ref{fig:magnitude} shows the case where the magnitude of the error, $e_{\nu k_{ii}}(j\omega)$, is higher than the magnitude of the estimator, $e_{est}(j\omega)$, at high frequencies. Figure \ref{fig:poles} depicts the case where the poles of the error at lower frequencies than the poles of the estimator. The third counterexample in Figure \ref{fig:zeros} illustrates the case where the slope of the error is lower than the slope of the estimator. The slope of the error is related with the number of zeros at low frequency. The error estimator has two zeros at the expansion point $s_e$. Thus, it needs to be shown that the actual error at least two zeros at the expansion point, which is placed close to zero. Figure \ref{fig:good} shows graphically that if the other three counterexamples are not true, the error estimator is an upper bound of the FRF of the of the real error for the whole frequency range.

	\begin{figure}
		\centering
		\subfigure[Counterexample: the magnitude of the error is higher than the magnitude of the estimator at high frequencies ]{\label{fig:magnitude}\def \svgwidth{0.48\textwidth}{\small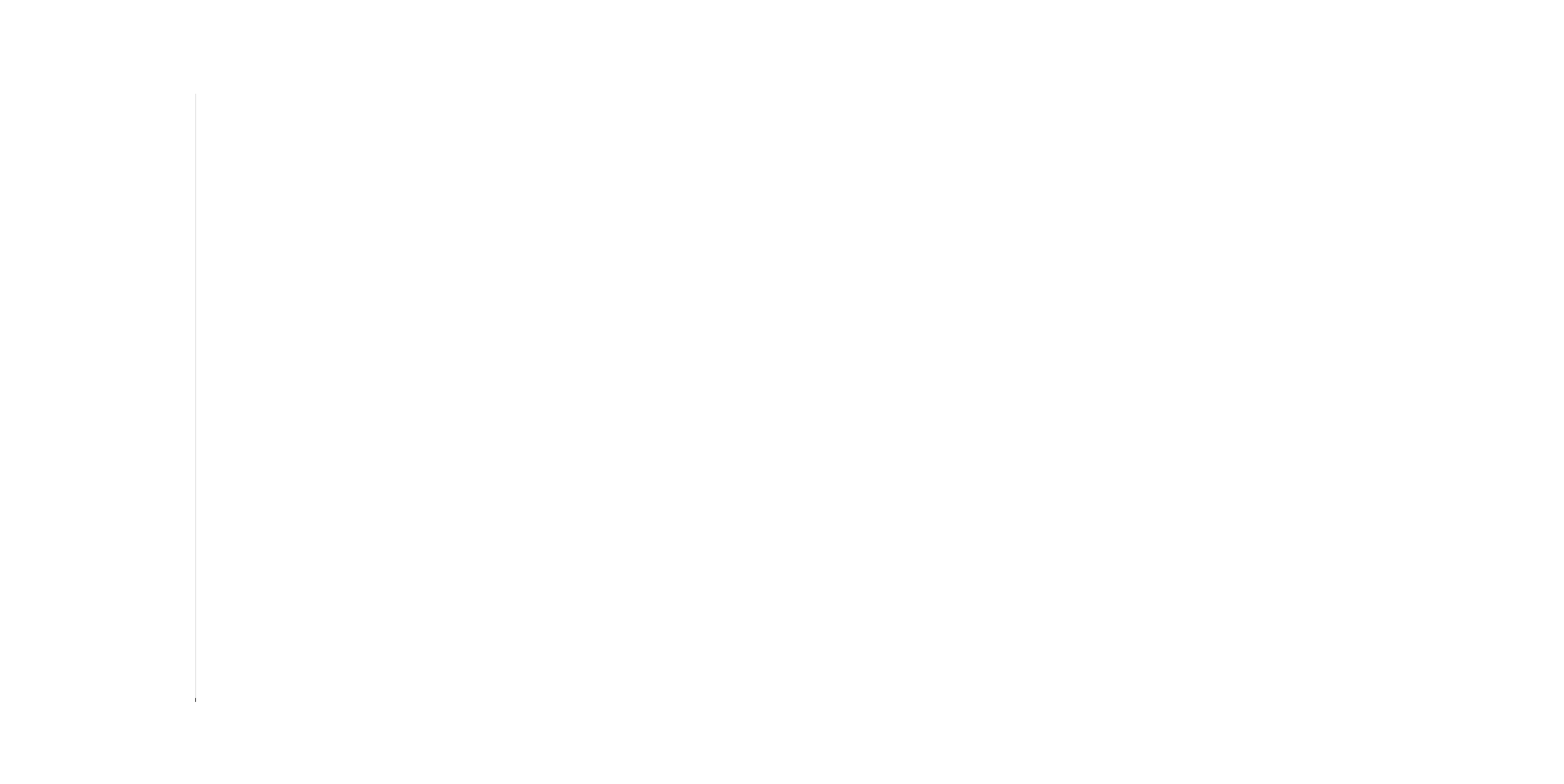}}\hfill
		\subfigure[Counterexample: the poles of the error at lower frequencies than the poles of the estimator]{\label{fig:poles}\def \svgwidth{0.48\textwidth}{\small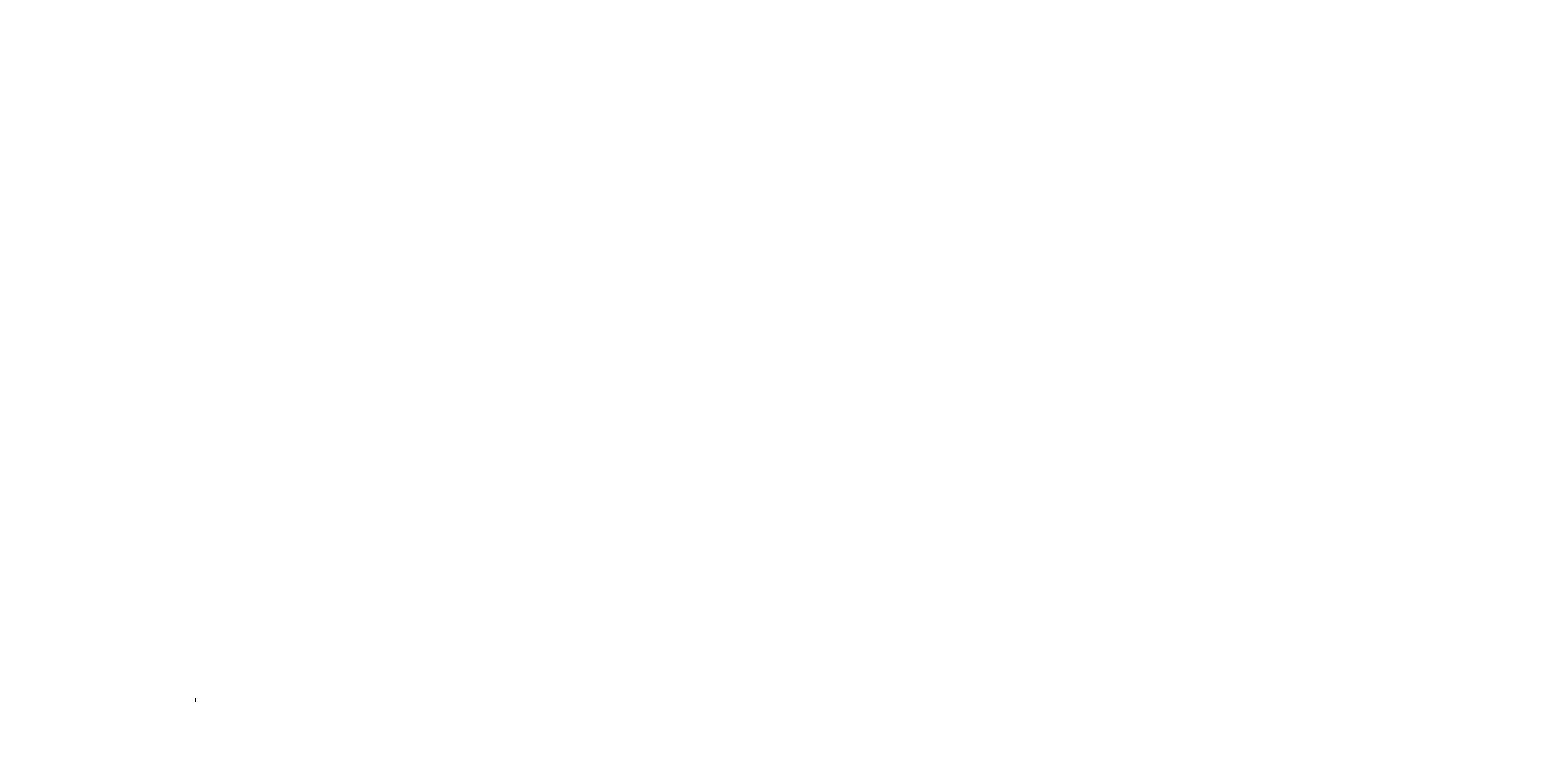}}
		\subfigure[Counterexample: the slope of the error is lower than the slope of the estimator]{\label{fig:zeros}\def \svgwidth{0.48\textwidth}{\small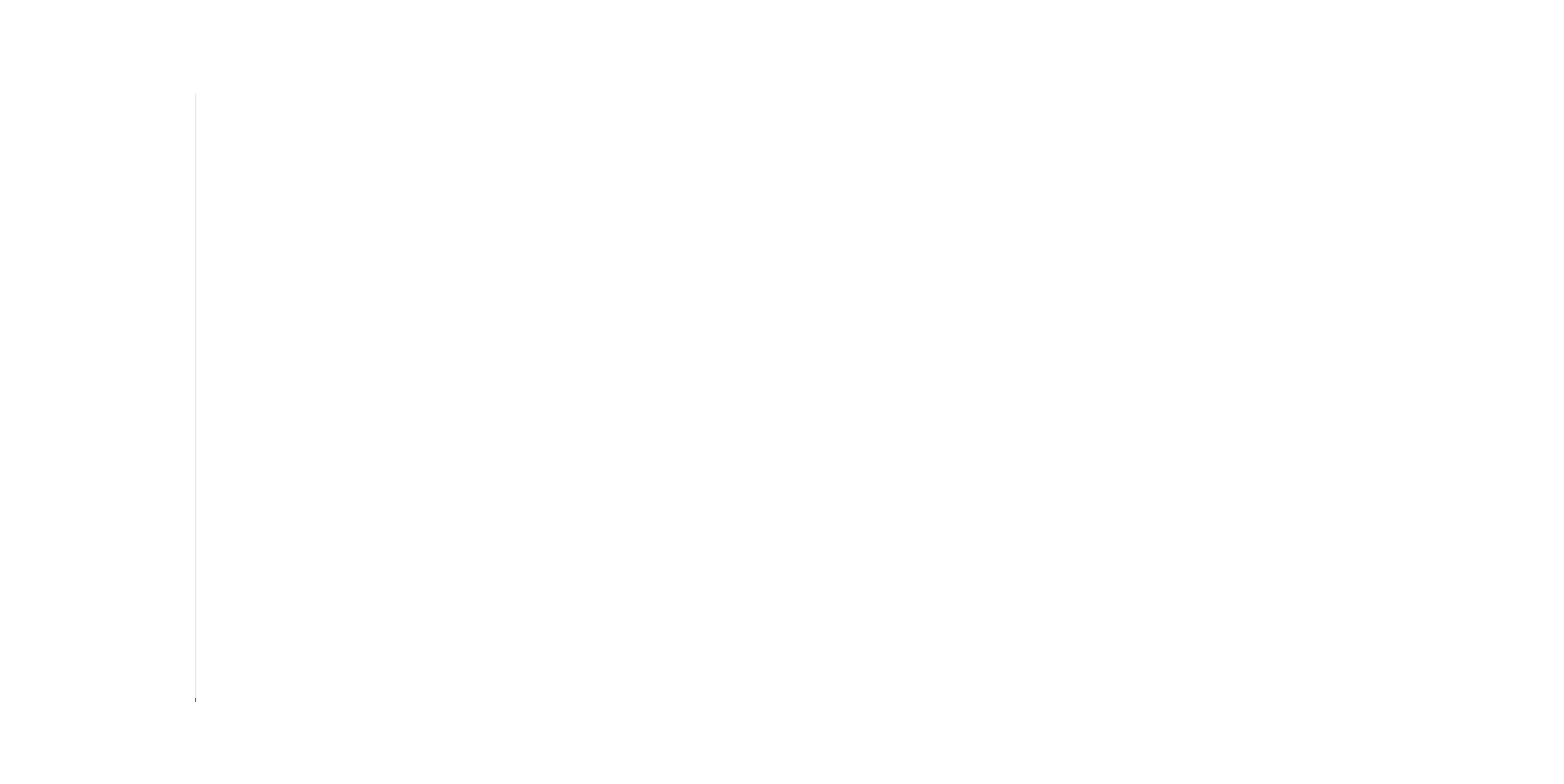}}\hfill
		\subfigure[The error estimator is an upper bound of the real error for all frequencies ]{\label{fig:good}\def \svgwidth{0.48\textwidth}{\small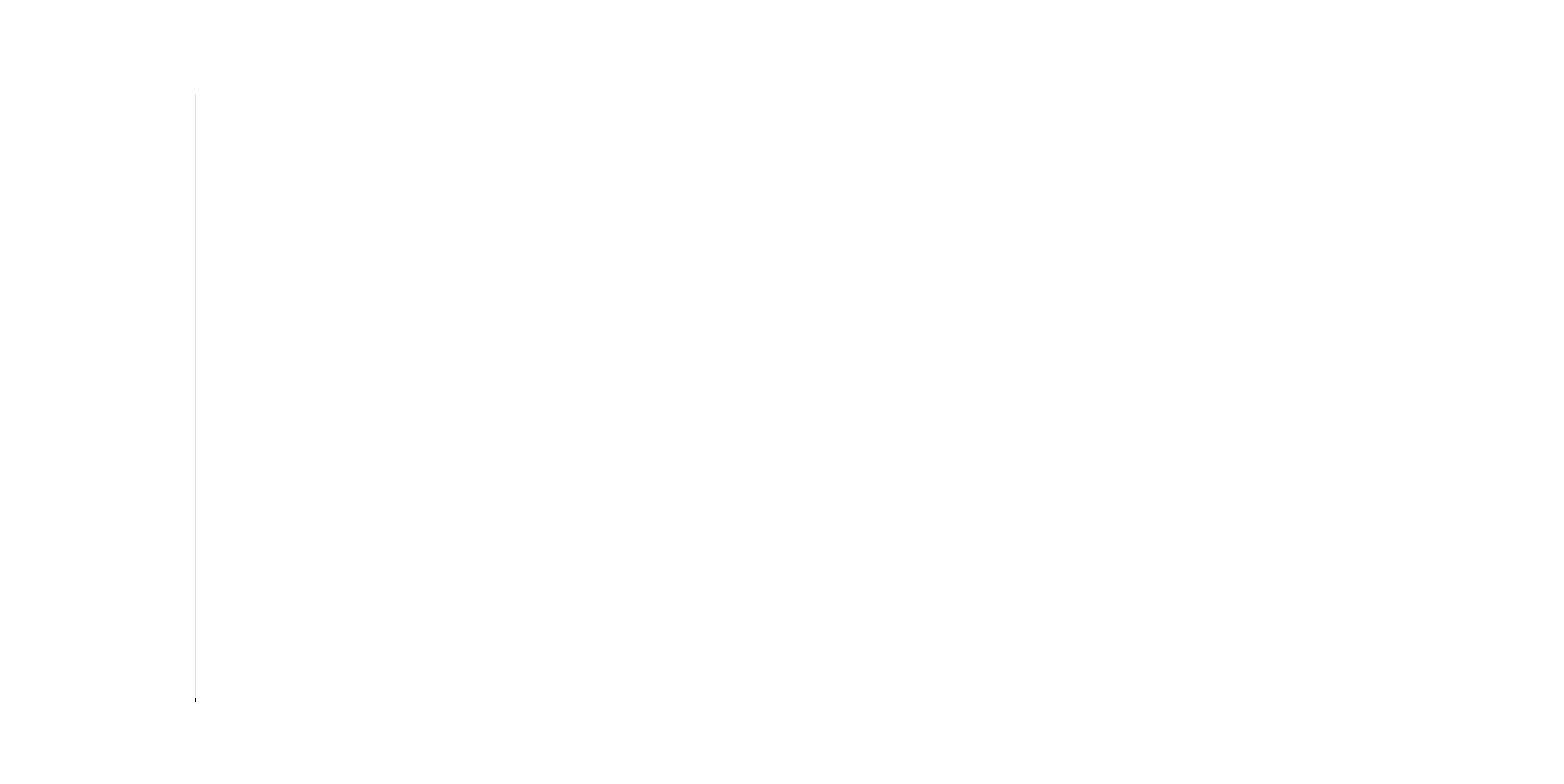}}
		\caption
		{\small Magnitude of the error of the estimator $e_{est}(j\omega)$ compared to theoretically possible magnitude of the error $e_{\nu k_{ii}}(j\omega)$} 
		\label{fig:couter}
	\end{figure}
	
	Therefore, it needs to be proven that the counterexamples of Figure \ref{fig:couter} are not possible following the next steps:
	
	\begin{enumerate}[label=Step \arabic*:]
		\item $\lim\limits_{\omega \to \infty}\rvert e_{est}(j \omega) \rvert> \lim\limits_{\omega \to \infty}\rvert e_{\nu k_{ii}}(j \omega) \rvert$
		\item All the poles of the actual error $e_{\nu k_{ii}}(j \omega)$ are at a higher frequency than the ones of the error estimator $e_{est}(j \omega)$
		\item $e_{\nu k_{ii}}(j \omega)$ has at least two zeros at $\omega$ close to zero
	\end{enumerate}
	
	The first condition refers to the magnitude of the FRF at infinity, which is associated to the Markov parameters. From the definition of the error estimator $\lim\limits_{\omega \to \infty}\left\lvert e_{est}(j\omega) \right\rvert = 1$. The magnitude of $e_{\nu k_{ii}}(j \omega)$ at infinity can be expressed as 
	\begin{equation}
	\lim\limits_{\omega \to \infty}\left\lvert e_{\nu k_{ii}}(j\omega) \right\rvert = \lim\limits_{\omega \to \infty}\left\lvert  \frac{\tilde{h}_{\nu k_{ii}}(j \omega)}{\tilde{h}_{\nu_{ii}}(j \omega)}  -1\right\rvert = \left\lvert  \lim\limits_{\omega \to \infty}\frac{\tilde{h}_{\nu k_{ii}}(j \omega)}{\tilde{h}_{\nu_{ii}}(j \omega)}  -1\right\rvert < 1
	\label{eq:condition}
	\end{equation}

	In order to evaluate the limit of the error, the limit $\lim\limits_{\omega \to \infty}\frac{\tilde{h}_{\nu k_{ii}}(j \omega)}{\tilde{h}_{\nu_{ii}}(j \omega)}$ needs to be evaluated. 
	
	\begin{equation}
		\lim\limits_{\omega \to \infty}\frac{\tilde{h}_{\nu k_{ii}}(j \omega)}{\tilde{h}_{\nu_{ii}}(j \omega)} = \frac{\frac{1}{j \omega} \frac{\tilde{b}^2}{\tilde{e} }}{\frac{1}{j \omega}\sum_{k=1}^{\nu}b_k^2} = \frac{\frac{\tilde{b}^2}{\tilde{e} }}{\sum_{k=1}^{\nu}b_k^2}
		\label{eq:limit_magnitude}
	\end{equation}
	where $\frac{\tilde{b}^2}{\tilde{e} }$ can be expressed according to Equation \eqref{eq:FRF_kr} as
	\begin{equation}
	\frac{\tilde{b}^2}{\tilde{e} } = \frac{\left(\sum_{k=1}^{\nu} \frac{b_{k}^2}{\omega_{k}+s_e}\right)^2}{ \sum_{k=1}^{\nu} \frac{b_{k}^2}{(\omega_{k}+s_e)^2} } 
	\end{equation}
	
	Equation \eqref{eq:limit_magnitude} shows that the limit $\lim\limits_{\omega \to \infty}\frac{\tilde{h}_{\nu k_{ii}}(j \omega)}{\tilde{h}_{\nu_{ii}}(j \omega)}$ is a positive real number. Therefore, the condition on the magnitude of the error of Equation \eqref{eq:condition} can be expressed as 
	\begin{equation}
		0 < \lim\limits_{\omega \to \infty}\frac{\tilde{h}_{\nu k_{ii}}(j \omega)}{\tilde{h}_{\nu_{ii}}(j \omega)} <2
		\label{eq:condition_2}
	\end{equation}
	In fact, an upper bound for the term $\frac{\tilde{b}^2}{\tilde{e} }$ can be found as follows
	\begin{equation}
	 \frac{\tilde{b}^2}{\tilde{e} } < \frac{\left(\sum_{k=1}^{\nu} \frac{b_{k}^2}{\omega_{k}+s_e}\right)^2}{\frac{1}{\omega_1+s_e} \sum_{k=1}^{\nu} \frac{b_{k}^2}{\omega_{k}+s_e} } =  \sum_{k=1}^{\nu} \frac{\omega_{1}+s_e}{\omega_{k}+s_e} b_k^2 < \sum_{k=1}^{\nu}b_k^2 
	\end{equation}
	considering that $\omega_1 \leq \omega_k$ for all $k \in [1, \nu]$. The upper bound for the term $\frac{\tilde{b}^2}{\tilde{e} }$ implies that the condition of Equation \eqref{eq:condition_2} is satisfied. Therefore, the magnitude of the actual error is bounded by 1, as stated in Equation \eqref{eq:limit_magnitude}. This completes the first step of the proof, showing that the magnitude of the error estimator is an upper bound of the magnitude of the actual error, $e_{\nu k_{ii}}(j \omega)$, at infinity.

	The second step in this proof is ensuring that the poles of the actual error are at a higher frequency than the ones of the error estimator. Let $e_{\nu k_{ii}}(s)$ the transfer function of the actual error, where $s \in \mathbb{C}$. The error estimator, $e_{est}(s)$, defined in Equation \eqref{eq:error_bound} has two poles at $\omega_m$. The poles of the error $e_{\nu k_{ii}}(s)$ are the poles of $\tilde{h}_{\nu k_{ii}}(s)$, the poles $\tilde{h}_{\nu_{ii}}(s)$, and the zeros of $\tilde{h}_{\nu_{ii}}(s)$, which need to be at a lower frequency than $\omega_m$. 
	
	The poles of $\tilde{h}_{\nu_{ii}}(s)$ are $\omega_1 \dots \omega_\nu$. By definition, these poles are at a higher frequency than $\omega_m$. 
	
	The zeros of $\tilde{h}_{\nu_{ij}}(s)$ are real numbers, as discussed in the beginning of this proof. Furthermore, it needs to be shown that these zeros are at higher frequencies than $\omega_m$, which is proven by contradiction. Assume that $\omega_0$ is a zero of the transfer function, i.e.\ $\tilde{h}_{\nu_{ii}}(s=-\omega_0)=0$, such that $0<\omega_0 <\omega_1$. Substituting this into Equation \eqref{eq:remaining_transfer}, the transfer function is 
	\begin{equation}
	\tilde{h}_{\nu_{ii}}(s=-\omega_0) = \sum_{k=1}^{\nu}\frac{b_k^2}{(-\omega_0 + \omega_{k})}
	\end{equation}
	However, given that $0\leq\omega_1 \leq\omega_{\nu}$ the all the terms of the summation are strictly positive. Thus,  $\tilde{h}_{\nu_{ii}}(s=-\omega_0)\neq 0$, reaching a contradiction. Therefore, it is shown that the zeros of the transfer function of the reduced system $\tilde{h}_{\nu_{ii}}(s)$ are at a higher frequency than $\omega_m$. 
	
	The pole $\tilde{h}_{\nu k_{ii}}(s)$, $\omega_{\nu k_{ii}}$, can be written as
	\begin{equation}
	\omega_{\nu k_{ii}} = \frac{\tilde{a}}{\tilde{e}} = \frac{\sum_{k=1}^{\nu} \frac{b_{k}^2}{(\omega_{k}+s_e)^2}\omega_{k}}{\sum_{k=1}^{\nu} \frac{b_{k}^2}{(\omega_{k}+s_e)^2}}
	\label{eq:zero_TF}
	\end{equation}
	which can be understood as a weighted summation, where the weight factors are $\frac{b_{k}^2}{(\omega_{k}+s_e)^2}$. This shows that the poles of the transfer function satisfy $\omega_1 \leq \omega_0 \leq \omega_r $. The location of $\omega_0$ depends on the static gains, \ $b_{k}$. The more controllable and observable a mode $\omega_k $ is, the closer $\omega_0 $ is to $\omega_k$. Thus, the pole of the transfer function of the reduced system $\tilde{h}_{\nu k_{ii}}(s)$ is at a higher frequency than $\omega_m$. This concludes the second step, ensuring that the poles of the actual error $e_{\nu k_{ii}}(s)$ are at a higher frequency than the ones of the error estimator.

	The third step needs to ensure that the transfer function of the actual error, $e_{\nu k_{ii}}(s)$, has at least two zeros close to zero. The zeros of the actual error are related with the zeros $\tilde{h}_{\nu k_{ii}}(s) - \tilde{h}_{\nu_{ii}}(s)$. Due to the properties of moment matching Krylov reduction, the reduced-order system matches at least two moments of the transfer function at the expansion point, $s_e$, as explained by Antoulas \cite{Antoulas2005}. By choosing the expansion point sufficiently low, the $\tilde{h}_{\nu k_{ii}}(s) - \tilde{h}_{\nu_{ii}}(s)$ has two zeros close to zero. Therefore, the transfer function $e_{\nu k_{ii}}(s)$ has also two zeros close to zero. 
	
	This proof shows that the error $e_{\nu k_{ii}}( j \omega)$ has a bounded magnitude at infinity, the poles of the error are a higher frequency than the poles of the estimator, and the error has at least two zeros close to zero frequency. The combination of these three conditions proves that the magnitude of the error $e_{\nu k_{ii}}( j \omega)$ is bounded by the estimator $\left \rvert e_{est}(j\omega) \right \rvert$.

\end{proof}

The result of Theorem \ref{theorem:error} proposes an error bound of the term $e_{\nu k_{ii}}(j \omega)$ of Equation \eqref{eq:error_separation}. In order to estimate the reduction error, $e_{ii}(j \omega) $, an upper bound of the term $e_{\mu_{ii}}(j \omega)$ of Equation \eqref{eq:error_separation} is also required. 

\begin{fact}
	The magnitude of the error $e_{\mu_{ii}}(j \omega)$ defined in Equation \eqref{eq:error_separation} is bounded by 1, i.e.\
	\begin{equation}
	\left \rvert e_{\mu_{ii}}(j\omega) \right \rvert \leq 1
	\end{equation}
	for all $\omega \in [0 , \infty]$ such that $\omega \in \mathbb{R}$. 
	
	\label{fact:error_modal}
\end{fact}

\begin{proof}

	The FRF of the error $e_{\mu_{ii}}(j \omega)$ defined in Equation \eqref{eq:modal_error} of Fact \ref{fact:separation_error} can be expressed as 
	\begin{equation}
	e_{\mu_{ii}}(j \omega) = \frac{\tilde{h}_{\nu_{ii}}(j\omega )}{h_{ii}(j \omega )}
	\end{equation}
	considering that $\tilde{h}_{\nu_{ii}}(j\omega ) = h_{ii}(j\omega ) -\tilde{h}_{\mu_{ii}}(j\omega )$, as shown in Fact \ref{fact:r_modal}. In order to prove this fact, the two extreme conditions are considered. Firstly, the case where all modes in $\mathcal{V}_{\nu}$ are not observable and not controllable is analyzed. In this case, the magnitude of $\tilde{h}_{\nu_{ii}}$ is zero. Thus, the magnitude of the error $ \left \rvert e_{\mu_{ii}}(j \omega) \right \rvert$ is zero for all frequencies, as all the relevant modes describing the response of the system are contained in $\mathcal{V}_{\mu}$. Secondly, it is considered that all modes in $\mathcal{V}_{\mu}$ are not observable and not controllable. In this case, the original system and the system projected into $\mathcal{V}_{\nu}$  are the same, i.e.\ $\tilde{h}_{\nu_{ii}} = h_{ii}(j \omega )$. Thus, the magnitude of the error $ \left \rvert e_{\mu_{ii}}(j \omega) \right \rvert$ is 1, as none of the relevant modes describing the response of the system are included in $\mathcal{V}_{\mu}$. These two cases represent the two extreme conditions. For any other case, some of the modes $\mathcal{V}_{\mu}$ and $\mathcal{V}_{\nu}$ are observable and controllable, leading to a magnitude of the error between 0 and 1. Therefore, $\left \rvert e_{\mu_{ii}}(j\omega) \right \rvert \leq 1$ for all frequencies. 
\end{proof}

The results of Theorem \ref{theorem:error} and Fact \ref{fact:error_modal} state that an error bound of $e_{ii}(j \omega)$ for the KMS reduction error is 
\begin{equation}
	\left \rvert e_{ii}(j \omega) \right \rvert = \left \rvert e_{\mu_{ii}}(j\omega ) \right \rvert \left \rvert e_{\nu k_{ii}}(j\omega ) \right \rvert < \left \rvert e_{\nu k_{ii}}(j\omega) \right \rvert <  \frac{\omega^2 + s_e^2}{\omega ^2 + w_m^2} 
	\label{eq:error_KMS}
\end{equation}

Theorem \ref{theorem:error} shows a theoretical error bound for the KMS reduction, provided that the same inputs and outputs are considered. This error bound estimates a-priori the error and enables to choose how many modes need to be included in the KMS projection basis. Provided a frequency range of interest $[0, \omega_{max}]$, the error estimator determines that all modes under $\omega_{m}$ need to be included in the reduction basis so that the magnitude of the error of the reduced-order system, $\left \rvert e_{ii}(j \omega) \right \rvert$, does not exceed a certain value $\epsilon$. Section \ref{sec:results} illustrates the validity of the error estimator with numerical examples for different combinations of the inputs and outputs.   

\section{Parametric KMS reduction}

The KMS method creates a reduction basis that captures the most relevant information of the reduction system in the frequency range of interest. This section presents a method to extend the KMS basis to include the possibility to modify the values of the HTC in the reduced system. The system of Equation \eqref{eq:system_distributed} shows the parametric dependency as an affine representation. This affine formulation of the system can be interpreted as bilinear system. The term $\sum_{i=1}^{n_c}h_i\bm{D}_i \bm{x}(t)$ can be understood as an input of the system that depends linearly on the state. Bilinear systems are a special class of non-linear systems that are linear in the input and state, but not jointly linear in the input and state.

Before presenting the reduction approach, several considerations are required on the convection matrices $\bm{D}_i$ of Equation \eqref{eq:system_distributed}. Equation \eqref{eq:convection_FEM} defines the convection matrix of an element $e$, which is integrated numerically. The evaluation of the integral at the Gauss points leads to off-diagonal terms in $\bm{K}^e_{conv}$. This is called consistent convection matrix. However, it is customary to diagonalize this matrix in order to avoid numerical oscillations close to the convective boundary condition, as pointed out by Bruns \cite{BRUNS20072859}. Therefore, a diagonal convection matrix $\bm{D}_i$ is considered, with zeros for all the degrees of freedom (dof) not affected by the convective boundary condition $i$. The diagonal convection matrix $\bm{D}_i$ can be understood as a linear map $\mathcal{D}_i$, whose range has a dimension equal to the number of nodes in the convective boundary,  $n_{d_i}$.

Another important consideration for the development of the reduction method is determining the dimension of the subspace spanned by $(s_e\bm{E-\bm{A}})^{-1}\bm{D}_i$. The matrix $(s_e\bm{E-\bm{A}})^{-1}$ can be understood as a linear transformation, $\mathcal{A}$. Given that $(s_e\bm{E-\bm{A}})^{-1}$ is invertible, $\mathcal{A}$ is a bijective linear transformation, i.e.\ $\text{dim}(\text{range}((s_e\bm{E-\bm{A}})^{-1})) = n$. Thus,

\begin{equation}
\text{dim}(\text{range}(\mathcal{A} \circ \mathcal{D}_i)) = n_{d_i} 
\end{equation}  

The intersection of the subspaces spanned by the several convection matrices also needs to be discussed. In thermo-mechanical models of mechatronic systems, the different convective boundary conditions affect different surfaces. Thus, it can be stated that $\Gamma_i \cap \Gamma_j = 0 \: \forall i\neq j$. Therefore

\begin{equation}
\text{range}(\mathcal{D}_i) \cap \text{range}(\mathcal{D}_j) = 0 \: \forall i\neq j
\end{equation}

This can be extended to the range of $\mathcal{A} \circ \mathcal{D}_i$ as

\begin{equation}
\text{range}(\mathcal{A} \circ \mathcal{D}_i) \cap \text{range}(\mathcal{A} \circ \mathcal{D}_j) = 0 \: \forall i\neq j
\label{eq:intersection}
\end{equation}

Let $\bm{V}_{KMS}$ be the KMS reduction basis of the system without parametric convective boundary conditions. The basis $\bm{V}_{KMS}$  spans the linear subspace defined in Equation \eqref{eq:subspace_KMS}. All the states in the KMS are a linear combination of the columns of $\bm{V}_{KMS}$, i.e.\ $\bm{x} \in \text{span}(\bm{V}_{KMS})$.

In order to consider the consider the state-dependent inputs, $\bm{D}_i \bm{x}$, an iterative approach is applied. The states multiplied by the convection matrix, i.e.\  $\bm{D}_i\bm{V}_{KMS}$, provide a new set of inputs for the reduction. The matrix $\bm{D}_i\bm{V}_{KMS}$ can be added as further inputs to the Krylov subspace of Equation \eqref{eq:subspace_krylov}, providing a new set of vectors $\bm{V}^1_{i}$. This set of new basis vectors multiplied by the convection matrix, i.e.\ $\bm{D}_i\bm{V}^1_{i}$, provide new inputs. They can be added to the Krylov subspace delivering a new set of basis vectors $\bm{V}^2_{i}$. This iterative process can be expressed as
\begin{equation}
\begin{split}
\text{span}(\bm{V}_{KMS}) = \text{span}((s_e\bm{E-\bm{A}})^{-1}\bm{B}) + \text{span}(\bm{V}_{\mu}) \\
\bm{V}^1_{i} = (s_e\bm{E-\bm{A}})^{-1}\bm{D}_i\bm{V}_{KMS} \\
\dots\\
\bm{V}^k_{i} = (s_e\bm{E-\bm{A}})^{-1}\bm{D}_i\bm{V}^{k-1}_i\\
\dots\\
\bm{V}^{n_{me}}_{i} = (s_e\bm{E-\bm{A}})^{-1}\bm{D}_i\bm{V}^{n_{me}-1}_i
\end{split}
\label{eq:bilinear}
\end{equation}
where $n_{me}$ is the number of iterations. The sum of the subspaces spanned by the matrices $\bm{V}_{KMS}$, $\bm{V}^1_{i}$, $\dots$, $\bm{V}^k_{i}$, $\dots$, $\bm{V}^{n_{me}}_{i}$ is the reduced subspace $\mathcal{V}_{i}$ 

\begin{multline}
\mathcal{V}_{i} = \text{span} \{\bm{V}_{KMS}\}  + \text{span}\{ (s_e\bm{E-\bm{A}})^{-1}\bm{D}_i\bm{V}_{KMS} \} + \dots \\+ \text{span} \{ \left ( (s_e\bm{E-\bm{A}})^{-1}\bm{D}_i \right )^{n_{me}-1}\bm{V}_{KMS} \} 
\end{multline} 

Comparing the definition of $\mathcal{V}_{i}$ with the definition of the Krylov subspace, it can be stated that 

\begin{equation}
\mathcal{V}_{i} = \mathcal{K}_{n_{me}}  \{(s_e\bm{E-\bm{A}})^{-1}\bm{D}_i, \bm{V}_{KMS} \}
\end{equation}  

For the creation of the subspace $\mathcal{V}_{i}$, $n_{me} < n_{d_i}$ terms are selected. This process of adding new vectors to the reduction basis $\bm{V}_{KMS}$ can be understood as a Krylov reduction of the subspace spanned by $(s_e\bm{E-\bm{A}})^{-1}\bm{D}_i$. The reduced subspace can be calculated for all the $n_c$ convective boundary  conditions. Given $\mathcal{V}_{i}$ and $\mathcal{V}_{j}$ for two different convective boundary conditions, it can be stated that 
\begin{equation}
\mathcal{V}_{i} \cap \mathcal{V}_{j} = \text{span}(\bm{V}_{KMS})
\end{equation}
which comes from the properties of the convection matrices shown in Equation \eqref{eq:intersection}. The reduction subspace $\mathcal{V}_{p}$ is the addition of subspaces $\mathcal{V}_{i}$ as 
\begin{equation}
\mathcal{V}_{p} =  \sum_{i=1}^{n_c} \mathcal{V}_{i} = \text{span}(\bm{V})
\label{eq:subspace_bilinear}
\end{equation}
where $\bm{V}$ is an orthonormal basis of the subspace $\mathcal{V}_{p}$. 

The KMS reduction considers certain modes of the system, $\bm{V}_{\mu}$, in order to form the projection basis. The KMS method selects modes up to a certain frequency. The error estimator ensures that the error to the original system remains below a given tolerance for the frequency range of interest. However, the system matrix changes with the introduction of the convective boundary conditions, as shown in Equation \eqref{eq:system_distributed}. Thus, the eigenvalues and eigenvectors of the system change. It is required to determine if more modes need to be considered in order to ensure that the error remains bounded below a certain $\epsilon$ in the frequency range of interest. 

In order to investigate this, the Weyl's Inequality Theorem is presented as a preliminary result. Theorem \ref{theorem:weyl} bounds the eigenvalues of the sum of two matrices $\bm{M} + \bm{N}$ given the eigenvalues of $\bm{M}$ and $\bm{N}$. 

\begin{theorem}{Weyl's Inequality}
	Let $\bm{M}$, $\bm{N}$ be symmetric matrices of dimension $\mathbb{R}^{n}$ and let $\bm{S} = \bm{M} + \bm{N}$. Let $\alpha_1 \geq \alpha_2 \geq \dots \alpha_n$, $\beta_1 \geq \beta_2 \geq \dots \beta_n$, and $\gamma_1 \geq \gamma_2 \geq \dots \gamma_n$ be the eigenvalues of the matrices $\bm{M}$, $\bm{N}$, and $\bm{S}$ respectively. Then,
	
	\begin{equation}
	\gamma_j  \leq \alpha_i + \beta_{j-i+1} \:\:\:\: \forall i \leq j 
	\end{equation}
	\label{theorem:weyl}
\end{theorem} 

\begin{proof}
	See Bhatia \cite{bhatia97} Chapter 3. 
\end{proof}

The result of Theorem \ref{theorem:weyl} can be used to analyze the properties of the eigenvalues of the system of Equation \eqref{eq:system_distributed} with increasing values of the HTC $h_i$.  

\begin{theorem}
	Let the system matrix be $\bm{A}^l_d = \bm{A} +  \sum_{i=1}^{n_c}h^l_i\bm{D}_i$, where $h^l_i\geq0$ is a sample of the $i$th parameters defining the $n_c$ convective boundary conditions. Let $0 \geq \alpha^l_1 \geq \alpha^l_2 \geq \dots \alpha^l_n $ be the eigenvalues of $\bm{A}^l_d$. Let the system matrix be $\bm{A}^m_d = \bm{A} +  \sum_{i=1}^{n_c}h^l_i\bm{D}_i$, where $h^m_i$ is another sample of the $i$ parameters such that $h^m_i \geq h^l_i\geq 0$. Let $0 \geq \alpha^m_1 \geq \alpha^m_2 \geq \dots \alpha^m_n $ be the eigenvalues of $\bm{A}^m_d$. Then, $\alpha^m_j \leq \alpha^l_j$ for $j =1 \dots n$. 
	\label{theorem:modes}
\end{theorem}

\begin{proof}
	The system matrices $\bm{A}$ and $\bm{D}_i$ are negative semi-definite, i.e.\ all the eigenvalues are real and non-positive. The result of Theorem \ref{theorem:weyl} can be particularized for $i=j$ as
	\begin{equation}
	\gamma_j  \leq \alpha_j + \beta_{1} 
	\end{equation} 
	
	Additionally, considering a negative semi-definite matrix $\bm{N}$ it can be stated that 
	\begin{equation}
	\gamma_j  \leq \alpha_j + \beta_{1} \leq \alpha_j
	\label{eq:negativeweyl}
	\end{equation} 
	as $\beta_{1} \leq 0$. Let $\delta h_i$ be the difference between $h^m_i$ and $h^l_i$ for $i = 1 \dots n_c$, i.e.\ $\Delta h_i = h^m_i - h^l_i$. From the statement of this theorem, $\Delta h_i \geq 0$. The system matrices can be expressed as $\bm{A}^m_d = \bm{A}^l_d + \sum_{i=1}^{n_c}\Delta h_i\bm{D}_i$. Applying the particularization of the Theorem \ref{theorem:weyl} to negative semi-definite matrices of Equation \eqref{eq:negativeweyl}, the following result is obtained
	\begin{equation}
	\alpha^m_j \leq \alpha^l_j 
	\end{equation}
	
\end{proof}

The result of Theorem \ref{theorem:modes} states that the higher the HTC, the more negative the eigenvalues of the system are. The physical interpretation of this theorem is that the higher the convective heat exchange with the surrounding, the faster the time constants of the system are.  On the other side, a system with low HTC evacuates less efficiently the heat and thus has slower time constants.

Theorem \ref{theorem:modes} can be particularized for the case that the initial HTC are equal to zero, i.e.\ $h^l_i = 0$ for all $i = 1 \dots n_c$. This is the case of the first iteration of the bilinearization, as shown in Equation \eqref{eq:bilinear}. The eigenvalues of the system without convective boundary conditions are less negative. Thus, the eigenfrequcies of the selected eigenmodes in $\bm{V}_{\mu}$ increase in magnitude after including the convection matrices. Therefore, the reduced system with proposed bilinearization remains valid in the frequency range of interest.  

The proposed parametric MOR approach extends the KMS projection basis to include the parametric dependency of the HTC. Therefore, the dimension of the reduced system increases with respect to the non-parametric reduced model. The dimension of the parametric model, $r_{p}$, depends on the dimension of the number of convective boundary conditions, $n_c$, the number of iteration, $n_{me}$, and the dimension of the non-parametric reduced system, $r$,  as
\begin{equation}
r_{p} = r(1+n_{me}n_c)
\label{eq:dimension_bilinear}
\end{equation}

\section{Thermo-mechanical coupling}

The KMS reduction constructs a projection basis that reduces the thermal system of Equation \eqref{eq:system_distributed}. The thermal model determines the transient behavior of the system. However, the main output of interest of the thermo-mechanical model is the mechanical deformation of the structure. Therefore, a reduction basis for the system of Equation \eqref{eq:coupling_state} and \eqref{eq:LTI_output_mech} needs to be defined. 

The mechanical model only considers the quasi-static response of the structure. Therefore, the reduced subspace for the mechanical system can be constructed using a Krylov subspace with one moment and one expansion point at a low frequency. The reduced system matches the static gain of the mechanical system. In order to construct the projection basis, the number of inputs of Equation \eqref{eq:coupling_state} needs to be considered. The number of mechanical inputs is equal to the number of possible independent temperature distributions and the number of external loads. In principle, the number of independent temperatures is $n$, i.e.\ the dimension of the original thermal system. However, the reduced thermal model determines the temperature distribution. Thus, the possible temperature distributions are limited to a linear combination of the vectors of the reduction basis $\bm{V}$. All the possible linearly independent thermal body forces $\bm{F}_{th}$ are
\begin{equation}
	\bm{F}_{th} = \bm{K}_{th}\bm{V}
	\label{eq:coupling_V}
\end{equation}

Therefore, the state space representation of Equation \eqref{eq:coupling_state} can be expressed in terms of the reduced temperature state $\bm{\tilde{x}}$ instead of the full state $\bm{x}$ as
\begin{equation}
	\bm{A}_{mech}\bm{x}_{mech} = \begin{bmatrix}
	\bm{K}_{th}\bm{V} & \bm{F}_{ext}
	\end{bmatrix}\begin{bmatrix}
	\bm{\tilde{x}} \\ \bm{u}_{ext}
	\end{bmatrix}
	\label{eq:sys_mech}
\end{equation} 
where the number of mechanical inputs is the order of the thermal reduced system and the external forces. The reduction basis $\bm{V_{mech}}$ for the mechanical system can be calculated using the Krylov subspace method as
\begin{equation}
	\text{span}(\bm{V}_{mech}) = \mathcal{K}_r \{(s_e\bm{I}-\bm{A}_{mech})^{-1}, (s_e\bm{I}-\bm{A}_{mech})^{-1}\begin{bmatrix}
	\bm{K}_{th}\bm{V} & \bm{F}_{ext}
	\end{bmatrix}  \}
	\label{eq:Krylov_mech}
\end{equation} 
where the expansion point $s_e$ is chosen close to zero. The Krylov subspace matches exactly static response of the system of Equation \eqref{eq:sys_mech} for the different temperature distributions considered in the thermal system projected by $\bm{V}$. 

\section{Numerical results}
\label{sec:results}

This section illustrates the MOR methods developed in this paper with numerical examples. Figure \ref{fig:RTM} shows the geometry and FE-discretization of a 3-axis milling machine. The thermo-mechanical model represents the thermal response of the machine tool under internal heat sources, such as the heat dissipated in the linear drives, and external effects, such as fluctuations of the environmental temperature.

\begin{figure}[!h]	 	
	\includegraphics[width=0.35\textwidth]{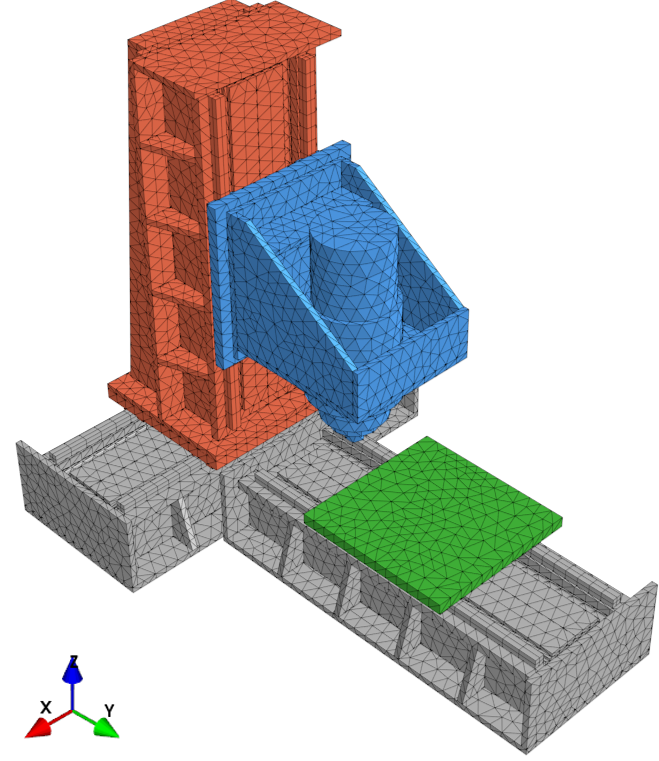}
	\centering
	\caption{Three axis milling machine}
	\label{fig:RTM}
\end{figure}

Figure \ref{fig:boundary_table} shows the geometry and FE-mesh of the table of the machine tool, which can move in Y-direction. The simple geometry under consideration leads to an original system with a small number of dofs, i.e. 4157 thermal dofs. The small size of the original system facilitates the evaluation of its thermal response in order to compare it with the reduced system. The convective boundary conditions of the machine tool table are separated into two regions, as shown in Figure \ref{fig:boundary_table}. Each of the convective boundary conditions has an independent value of the HTC, namely $\text{HTC}_{\text{top}}$ and $\text{HTC}_{\text{bot}}$. Additionally, the bottom of the table is exposed to a heat flux, representing the heat losses of a linear drive.

\begin{figure}
	\centering
	\subfigure[The top of the table]{\label{fig:a}\includegraphics[width=0.25\textwidth]{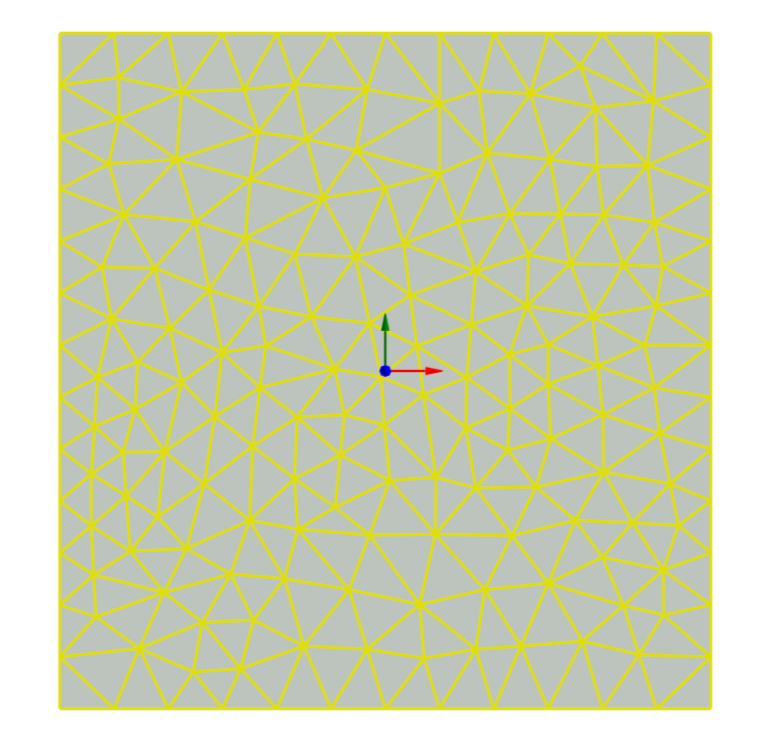}}
	\subfigure[Bottom of the table]{\label{fig:b}\includegraphics[width=0.25\textwidth]{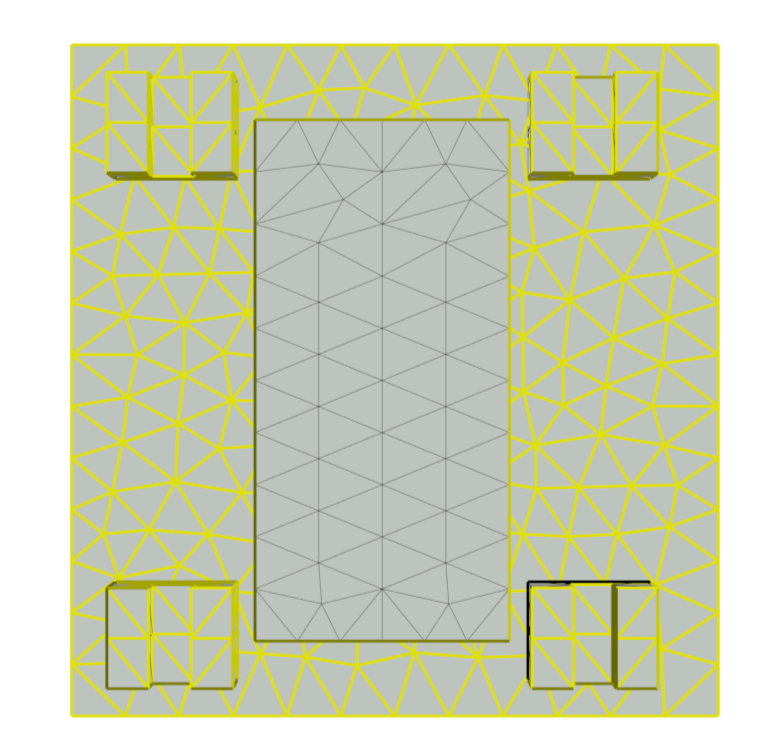}}
	\caption{Convective boundary conditions of the thermal model of the table, i.e.\ Y-axis, of the machine tool of Figure \ref{fig:RTM}}
	\label{fig:boundary_table}
\end{figure}

The a priori error estimator of Equation \eqref{eq:error_KMS} facilitates the selection of the parameters for the KMS reductions. Given a maximum error between the reduced and original system, $\epsilon$, and a frequency range of interest, $\omega \in [0, \omega_{max}]$, the error estimator provides the value of $\omega_m$. The reduction basis includes all eigenvectors with eigenfrequencies below the frequency $\omega_m$, i.e.\ $\omega_{\mu} < \omega_m$. For the FE thermal model under consideration, the following parameters are chosen: 

\begin{itemize}
	\item $\epsilon = $  0.05
	\item $\omega_{max} = $ 0.01 rad/s
\end{itemize}

The maximum frequency of interest, $\omega_{max}$, determines the frequency range containing the most relevant part of the response to the considered thermal inputs. For the thermal model of Figure \ref{fig:boundary_table}, the value of $\omega_m $ is 0.04367 rad/s. This results in including 97 modes in the KMS projection basis $\bm{V}_{KMS}$. The KMS reduction requires additionally an expansion point at a low frequency in order to match the steady state response of the system. For this thermal model, the reductions basis $\bm{V}_{KMS}$ considers an expansion point $s_e = 10^{-8}$ rad/s. The bilinear reduction extends the basis $\bm{V}_{KMS}$ with the information about the convective boundary conditions. For the investigated thermal model, the number of parametric convective boundary conditions, $n_c$, is 2. The number of iterations for the bilinear reduction ($m_d$) is set to 2. The resulting reduced model has 496 dof, which is considerably smaller than the dimension of the original system.  

In order to evaluate the performance of the reduction method, the transfer functions of the thermal response of the original and reduced systems are compared. The input of the transfer function is the fluctuation of the environmental temperature. The output of the transfer function is the temperature measured at the drive of the Y-axis. The transfer functions are evaluated for different values of the parameters describing the convective boundary conditions. The error between the reduced and the original system is calculated between $10^{-5}$ and 1 rad/s according to Equation \eqref{eq:error_definition_ij} for the considered system input and output. Figure \ref{fig:error_bilinear_FRF} depicts the relative error between original and the system reduced by means of the parametric KMS method. The error is negligibly small at low frequencies, as the chosen expansion point matches the steady state response. At high frequencies, the error increases reaching values close to 1. The reduced system succeeds in reproducing the thermal response of the system in the frequency range of interest, i.e.\  up to 0.01 rad/s. Additionally, Figure \ref{fig:error_bilinear_FRF} shows that the error estimator of Equation \eqref{eq:error_KMS} is an upper bound of the relative error for all frequencies.

\begin{figure}[!h]
	\centering
	\def \svgwidth{\textwidth} 
	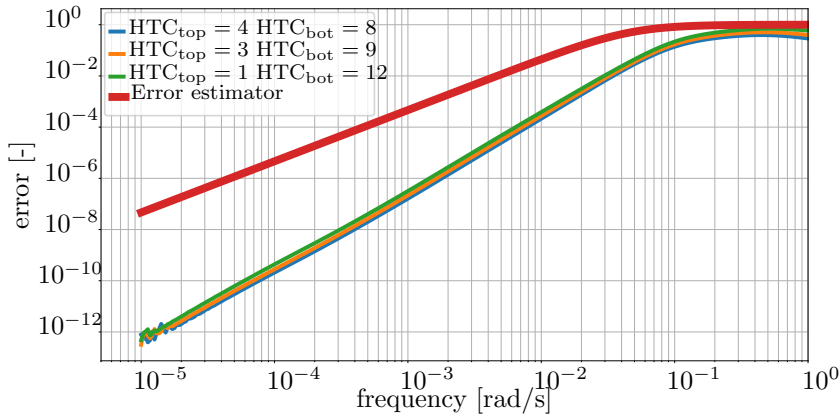 
	\caption{FRF of the relative error between reduced and original system for different combinations of values of $\text{HTC}_{\text{top}}$ and $\text{HTC}_{\text{bot}}$ in $\frac{\text{W}}{\text{m}^2\text{K}}$. The input is the environmental temperature and the output the mean temperature of the drive. Comparison with the error estimator of Equation \eqref{eq:error_KMS}}
	\label{fig:error_bilinear_FRF}
\end{figure}

Figure \ref{fig:error_bilinear_FRF} shows the reduction errors for one single combination of the input, environmental temperature, and output, mean temperature at the drive. Figure \ref{fig:error_bilinear_FRF_inputs} extends the evaluation of the relative error of the reduced and original system to other input and output combinations, fixing the values of $\text{HTC}_{\text{top}}$ and $\text{HTC}_{\text{bot}}$ to 4 and 8 $\frac{\text{W}}{\text{m}^2\text{K}}$ respectively. For the considered numerical example, the relative errors for different inputs and output combinations show more variability than the errors for Figure \ref{fig:error_bilinear_FRF}. The proposed error estimator remains an upper bound of the relative error for the transfer functions of the reduction error.    

\begin{figure}[!h]
	\centering
	\def \svgwidth{\textwidth} 
	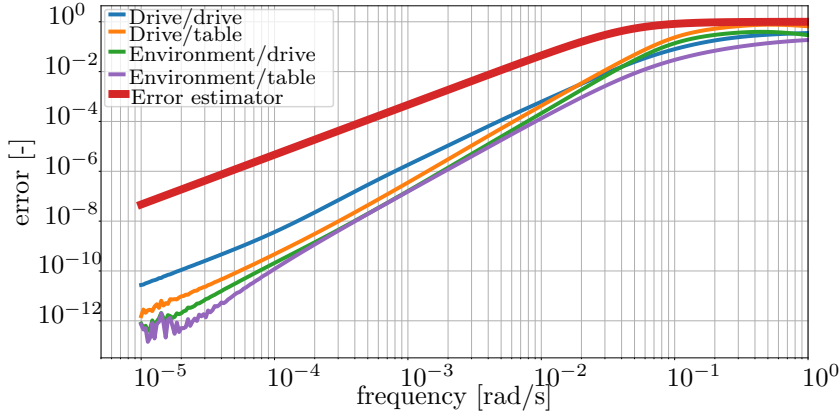 
	\caption{FRF of the relative error between reduced and original system for different combinations of inputs (heat flow at the drives and environment) and outputs (mean temperature at drive and table). The values of $\text{HTC}_{\text{top}}$ and $\text{HTC}_{\text{bot}}$ are 4 and 8 $\frac{\text{W}}{\text{m}^2\text{K}}$ respectively. Comparison with the error estimator of Equation \eqref{eq:error_KMS}}
	\label{fig:error_bilinear_FRF_inputs}
\end{figure}

In order to study the thermal response of the reduced system, it is interesting to compare the eigenvalues of the reduced and the original systems. Let $\omega_i$ be the $i$th eigenvalue of the system of Equation \eqref{eq:system_distributed} and let $\tilde{\omega}_i$ be the $i$th eigenvalue of the system projected to the subspace spanned by the projection basis $\bm{V}$ of Equation \eqref{eq:subspace_bilinear}. The relative error for the $i$th eigenvalue can be defined as 

\begin{equation}
	e_i = \left\vert\frac{\tilde{\omega}_i-\omega_i }{\omega_i}\right \vert
\end{equation}

Figure \ref{fig:error_bilinear_modes} shows the relative error of the first 90 eigenvalues between the original and reduced system. The eigenvalues of the systems lie in the frequency range of interest, ranging from $7 \cdot 10^{-5}$ to $0.042$ rad/s. The relative errors are evaluated for different combinations of the HTC, considering the same values of the HTC parameters as in Figure \ref{fig:error_bilinear_FRF}. The relative error of the eigenvalues increases at higher mode numbers, as illustrated in Figure \ref{fig:error_bilinear_modes}. Theoretically, if the values of the HTC are all zero, i.e.\ if no convective boundary conditions are applied, the KMS reduction basis matches exactly the eigenvalues of the original system. However, including values of the HTC different than zero results in a difference between the reduced and the original eigenvalues due to the bilinearization process. For the thermal FE model under consideration, the maximum relative error between the eigenvalues of the reduced and original system remain below $2 \cdot 10 ^{-5}$. Therefore, the reduced system with parametric convective boundary conditions succeeds in approximating the eigenvalues of the original system for different combinations of values of the HTC.

\begin{figure}[!h]	 
	\centering
	\def \svgwidth{\textwidth} 
	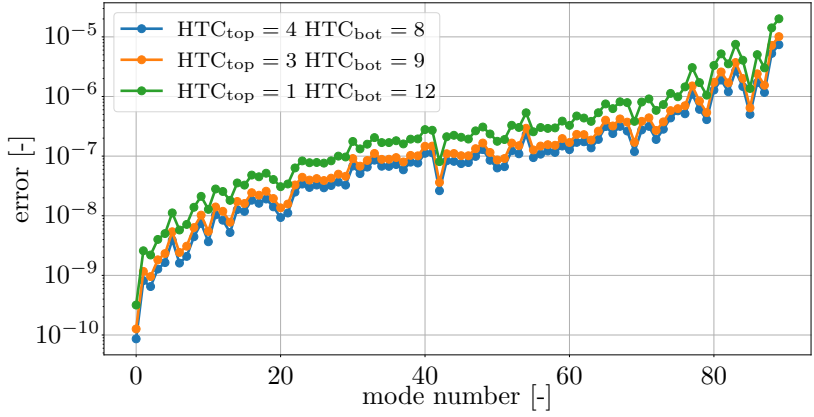 	
	\caption{Relative error between the first 90 eigenfrequencies of the reduced and original system for different combinations of values of $\text{HTC}_{\text{top}}$ and $\text{HTC}_{\text{bot}}$ in $\frac{\text{W}}{\text{m}^2\text{K}}$}
	\label{fig:error_bilinear_modes}
\end{figure}

The parametric KMS method creates a reduced thermal model, which reproduces the dynamics of the thermal behavior. However, the thermally induced structural deformation is the output of interest of the thermo-mechanical models of machine tools. In order to evaluate the mechanical response of the model of Figure \ref{fig:boundary_table}, the mechanical boundary conditions are defined as follows

\begin{itemize}
	\item The structure is fixed at the location of the guide carriages 
	\item The stiffness value is $10^{10}$ $\frac{\text{N}}{\text{m}}$ in all directions   
\end{itemize}

Equation \eqref{eq:Krylov_mech} defines the projection basis for the thermo-mechanical system. The expansion point $s_0$ for the Krylov subspace of the mechanical system is 30 rad/s. Starting from an original mechanical system with 12471 dofs, the coupling to thermal parametric reduced system results in a thermo-mechanical reduced system of dimension 532.   

In order to evaluate the accuracy of the reduction, the FRF of the thermo-mechanical the original and reduced system are compared. The thermal input is the fluctuation of the environmental temperature, which corresponds to the thermal input shown in Figure \ref{fig:error_bilinear_FRF}. The mechanical outputs are the linear displacements in X-, Y-, and Z-direction measured at the center of the machine tool table, where the workpiece is placed. The relative error between the two systems are calculated in the frequency range between $10^{-5}$ and 1 rad/s. Figure \ref{fig:mech_error_HTC} shows the FRFs of the relative error between the original and reduced systems for different values of the HTC, considering the displacement in Y-direction as the output of the FRF. Figure \ref{fig:mech_error_HTC} illustrates that the resulting relative errors between reduced and original system remain below 0.01 for up to the maximum frequency of interest, 0.01 rad/s. Thus, the thermo-mechanical coupled reduced system reproduces accurately the response of the original system for different values of the HTC. 

\begin{figure}
	\centering
	\def \svgwidth{\textwidth} 
	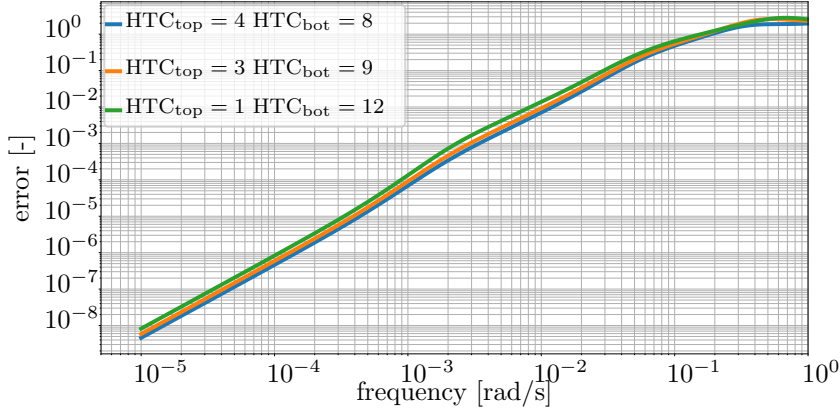	 	
	\caption{Relative error between the thermo-mechanical reduced and the original system in Y-direction for different combinations of values of $\text{HTC}_{\text{top}}$ and $\text{HTC}_{\text{bot}}$ in $\frac{\text{W}}{\text{m}^2\text{K}}$ }
	\label{fig:mech_error_HTC}
\end{figure}

Figure \ref{fig:mech_error} extends the evaluation to different outputs, namely the deviations in X- , Y-, and Z-direction.  The mechanical response of the system is evaluated for a value of the parameters $\text{HTC}_{\text{top}}$ and $\text{HTC}_{\text{bot}}$ of 4 and 8 $\frac{\text{W}}{\text{m}^2\text{K}}$ respectively. Figure \ref{fig:mech_error} illustrates that the  relative error of the mechanical response remains below 0.01 for the frequency range of interest. Therefore, this numerical example illustrates that the reduced thermo-mechanical system captures the response of the original system.

\begin{figure}
	\centering
	\def \svgwidth{\textwidth} 
	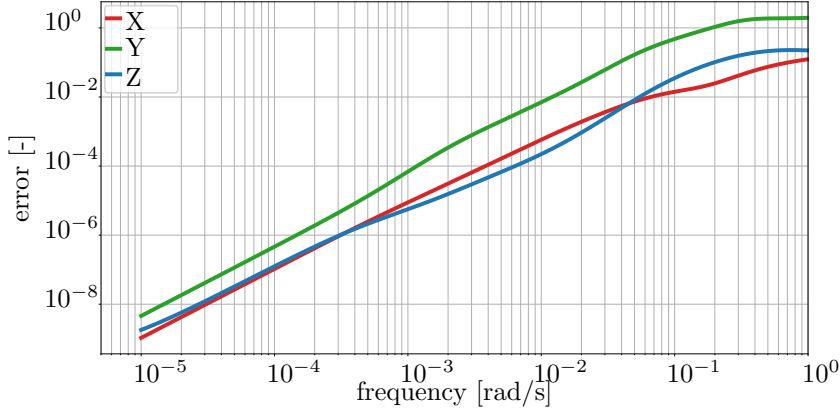	 	
	\caption{Relative error between the thermo-mechanical reduced and the original system in X-, Y-n and Z-direction. The values of $\text{HTC}_{\text{top}}$ and $\text{HTC}_{\text{bot}}$ are 4 and 8 $\frac{\text{W}}{\text{m}^2\text{K}}$ respectively.}
	\label{fig:mech_error}
\end{figure}

\section{Conclusions}

This work presents a MOR approach to deal with the parametric dependency of the convective boundary conditions in weakly coupled thermo-mechanical models. The KMS method proposes a reduction basis combining the information of the Krylov subspace basis with an expansion point at a low frequency and the thermal eigenmodes. The reduced system reproduces the steady state response of the system as well as the transient behavior up a to a certain frequency of interest. This method is especially suited for thermal models of machine tools and other similar mechatronic systems, which show a decaying amplitude of the response at higher frequencies. This work proposes an upper bound of the magnitude of the reduction error. The error bound provides an a-priori estimation of the KMS reduction error and determines the maximum eigenfrequecy to be included in the reduction basis so that the error remains below a certain value for the frequency range of interest.

Among the different parameters describing thermal models of machine tools, this work focuses on the HTC, which defines the convective heat exchange between the structure and the surrounding fluid media. This paper proposes a parametric MOR method to enable the modification of the value of the HTC after reduction. The KMS reduction basis is extended using the concept of system bilinearization. The numerical results show that the parametrically reduced model reproduces the thermal response of the original system for different values of the HTC. Additionally, the parametric reduction retains the thermal eigenfrequencies of the original model.  

The main objective of thermal error models of machine tools is the prediction of the thermally induced deviations. Therefore, the model needs to evaluate structural deformations associated to inhomogeneous, time-varying temperature distributions. This work presents a method to couple a reduced thermal and mechanical system. The developed coupling method creates a dedicated reduced mechanical system. On one hand, the mechanical system describes the response to any mechanical input, e.g.\ preloads or gravity, for any combination of the position of the axes. On the other hand, the mechanical system provides the response to any temperature distribution computed by the reduced system. The numerical results show that the reduced system succeeds in capturing the mechanical response of the original system for several values of the HTC and different outputs. The relative error of the reduction remains below 0.01 for the reduced thermo-mechanical system in the frequency range of interest. 

The developed MOR methods are implemented in a simulation software MORe \cite{MORe}, which is designed to facilitate the development of physical models of machine tools. Future work will investigate the theoretical error bounds of the KMS method for arbitrary expansion points as well as generalize the developed error estimator to different input and output combinations. Future research will develop thermo-mechanical models of the whole machine tool assembly with parametric boundary conditions. These models will evaluate the thermo-mechanical response of machine tools to varying conditions in the surrounding environment.  

\section{Acknowledgments}

The authors would like to thank the Swiss Innovation Agency (Innosuisse) and the Swiss National Science Foundation (SNSF) for their financial support.

\bibliography{mybibfile}

\begin{thebibliography}{10}
\expandafter\ifx\csname url\endcsname\relax
  \def\url#1{\texttt{#1}}\fi
\expandafter\ifx\csname urlprefix\endcsname\relax\def\urlprefix{URL }\fi
\expandafter\ifx\csname href\endcsname\relax
  \def\href#1#2{#2} \def\path#1{#1}\fi

\bibitem{2012_Mayr_thermal}
J.~Mayr, J.~Jedrzejewski, E.~Uhlmann, M.~A. Donmez, W.~Knapp, F.~H\"artig,
  K.~Wendt, T.~Moriwaki, P.~Shore, R.~Schmitt, C.~Brecher, T.~W\"urz,
  K.~Wegener, Thermal issues in machine tools, CIRP Annals - Manufacturing
  Technology 61~(2) (2012) 771 -- 791.
\newblock \href {https://doi.org/https://doi.org/10.1016/j.cirp.2012.05.008}
  {\path{doi:https://doi.org/10.1016/j.cirp.2012.05.008}}.

\bibitem{Bry1990}
J.~Bryan, International status of thermal error research (1990), CIRP Annals -
  Manufacturing Technology 39~(2) (1990) 645 -- 656.
\newblock \href {https://doi.org/https://doi.org/10.1016/S0007-8506(07)63001-7}
  {\path{doi:https://doi.org/10.1016/S0007-8506(07)63001-7}}.

\bibitem{ALTINTAS2005115}
Y.~Altintas, C.~Brecher, M.~Weck, S.~Witt, Virtual machine tool, CIRP Annals
  54~(2) (2005) 115 -- 138.
\newblock \href {https://doi.org/https://doi.org/10.1016/S0007-8506(07)60022-5}
  {\path{doi:https://doi.org/10.1016/S0007-8506(07)60022-5}}.

\bibitem{Benner2015}
P.~Benner, S.~Gugercin, K.~Willcox, A survey of projection-based model
  reduction methods for parametric dynamical systems, SIAM review 57~(4) (2015)
  483--531.
\newblock \href {https://doi.org/https://doi.org/10.1137/130932715}
  {\path{doi:https://doi.org/10.1137/130932715}}.

\bibitem{Herzog2018}
R.~Herzog, I.~Riedel, D.~Uci{\'{n}}ski, Optimal sensor placement for joint
  parameter and state estimation problems in large-scale dynamical systems with
  applications to thermo-mechanics, Optimization and Engineering 19~(3) (2018)
  591--627.
\newblock \href {https://doi.org/https://doi.org/10.1007/s11081-018-9391-8}
  {\path{doi:https://doi.org/10.1007/s11081-018-9391-8}}.

\bibitem{Lang2014}
N.~Lang, J.~Saak, P.~Benner, Model order reduction for systems with moving
  loads, at-Automatisierungstechnik 62~(7) (2014) 512--522.
\newblock \href {https://doi.org/https://doi.org/10.1515/auto-2014-1095}
  {\path{doi:https://doi.org/10.1515/auto-2014-1095}}.

\bibitem{Mian2013}
N.~S. Mian, S.~Fletcher, A.~Longstaff, A.~Myers, Efficient estimation by {FEA}
  of machine tool distortion due to environmental temperature perturbations,
  Precision Engineering 37~(2) (2013) 372 -- 379.
\newblock \href
  {https://doi.org/https://doi.org/10.1016/j.precisioneng.2012.10.006}
  {\path{doi:https://doi.org/10.1016/j.precisioneng.2012.10.006}}.

\bibitem{Weng2018}
L.~Weng, W.~Gao, Z.~Lv, D.~Zhang, T.~Liu, Y.~Wang, X.~Qi, Y.~Tian, Influence of
  external heat sources on volumetric thermal errors of precision machine
  tools, The International Journal of Advanced Manufacturing Technology 99~(1)
  (2018) 475--495.
\newblock \href {https://doi.org/10.1007/s00170-018-2462-3}
  {\path{doi:10.1007/s00170-018-2462-3}}.

\bibitem{Shi2018}
X.~Shi, K.~Zhu, W.~Wang, L.~Fan, J.~Gao, A thermal characteristic analytic
  model considering cutting fluid thermal effect for gear grinding machine
  under load, The International Journal of Advanced Manufacturing Technology
  99~(5) (2018) 1755--1769.
\newblock \href {https://doi.org/10.1007/s00170-018-2562-0}
  {\path{doi:10.1007/s00170-018-2562-0}}.

\bibitem{Panzer2010}
H.~Panzer, J.~Mohring, R.~Eid, B.~Lohmann, Parametric model order reduction by
  matrix interpolation, Automatisierungstechnik 58 (2010) 475--484.
\newblock \href {https://doi.org/10.1524/auto.2010.0863}
  {\path{doi:10.1524/auto.2010.0863}}.

\bibitem{Lee2017}
J.~Lee, M.~Cho, An interpolation-based parametric reduced order model combined
  with component mode synthesis, Computer Methods in Applied Mechanics and
  Engineering 319 (2017) 258 -- 286.
\newblock \href {https://doi.org/https://doi.org/10.1016/j.cma.2017.02.010}
  {\path{doi:https://doi.org/10.1016/j.cma.2017.02.010}}.

\bibitem{Amsallem2008}
D.~Amsallem, C.~Farhat, Interpolation method for adapting reduced-order models
  and application to aeroelasticity, Aiaa Journal - AIAA J 46 (2008)
  1803--1813.
\newblock \href {https://doi.org/10.2514/1.35374} {\path{doi:10.2514/1.35374}}.

\bibitem{BaurBenner2011}
U.~Baur, P.~Benner, A.~Greiner, J.~Korvink, J.~Lienemann, C.~Moosmann,
  Parameter preserving model order reduction for mems applications,
  Mathematical and Computer Modelling of Dynamical Systems 17~(4) (2011)
  297--317.
\newblock \href {https://doi.org/https://doi.org/10.1080/13873954.2011.547658}
  {\path{doi:https://doi.org/10.1080/13873954.2011.547658}}.

\bibitem{Phillips2003}
J.~R. {Phillips}, Projection-based approaches for model reduction of weakly
  nonlinear, time-varying systems, IEEE Transactions on Computer-Aided Design
  of Integrated Circuits and Systems 22~(2) (2003) 171--187.
\newblock \href {https://doi.org/10.1109/TCAD.2002.806605}
  {\path{doi:10.1109/TCAD.2002.806605}}.

\bibitem{BAI2006406}
Z.~Bai, D.~Skoogh, A projection method for model reduction of bilinear
  dynamical systems, Linear Algebra and its Applications 415~(2) (2006) 406 --
  425, special Issue on Order Reduction of Large-Scale Systems.
\newblock \href {https://doi.org/https://doi.org/10.1016/j.laa.2005.04.032}
  {\path{doi:https://doi.org/10.1016/j.laa.2005.04.032}}.

\bibitem{BREITEN2010443}
T.~Breiten, T.~Damm, Krylov subspace methods for model order reduction of
  bilinear control systems, Systems \& Control Letters 59~(8) (2010) 443 --
  450.
\newblock \href
  {https://doi.org/https://doi.org/10.1016/j.sysconle.2010.06.003}
  {\path{doi:https://doi.org/10.1016/j.sysconle.2010.06.003}}.

\bibitem{Benner2012}
P.~Benner, T.~Breiten, Interpolation-based \${\cal h}\_2\$-model reduction of
  bilinear control systems, SIAM Journal on Matrix Analysis and Applications
  33~(3) (2012) 859--885.
\newblock \href {http://arxiv.org/abs/https://doi.org/10.1137/110836742}
  {\path{arXiv:https://doi.org/10.1137/110836742}}, \href
  {https://doi.org/10.1137/110836742} {\path{doi:10.1137/110836742}}.

\bibitem{Bruns2015}
A.~Bruns, P.~Benner, Parametric model order reduction of thermal models using
  the bilinear interpolatory rational krylov algorithm, Mathematical and
  Computer Modelling of Dynamical Systems 21~(2) (2015) 103--129.
\newblock \href {https://doi.org/10.1080/13873954.2014.924534}
  {\path{doi:10.1080/13873954.2014.924534}}.

\bibitem{Bathe2006}
K.~Bathe, Finite element procedures in engineering analysis, Prentice Hall,
  1982.
\newblock \href {https://doi.org/https://doi.org/10.1002/nag.1610070412}
  {\path{doi:https://doi.org/10.1002/nag.1610070412}}.

\bibitem{Antoulas2005}
A.~C. Antoulas, Approximation of large-scale dynamical systems, Society for
  Industrial and Applied Mathematics, 2005.
\newblock \href {https://doi.org/https://doi.org/10.1137/1.9780898718713}
  {\path{doi:https://doi.org/10.1137/1.9780898718713}}.

\bibitem{Spescha_Diss}
D.~Spescha, Framework for efficient and accurate simulation of the dynamics of
  machine tools, Ph.D. thesis, TU Clausthal (Nov 2018).
\newblock \href {https://doi.org/10.21268/20181119-132644}
  {\path{doi:10.21268/20181119-132644}}.

\bibitem{BRUNS20072859}
T.~Bruns, Topology optimization of convection-dominated, steady-state heat
  transfer problems, International Journal of Heat and Mass Transfer 50~(15)
  (2007) 2859 -- 2873.
\newblock \href
  {https://doi.org/https://doi.org/10.1016/j.ijheatmasstransfer.2007.01.039}
  {\path{doi:https://doi.org/10.1016/j.ijheatmasstransfer.2007.01.039}}.

\bibitem{bhatia97}
R.~Bhatia,
  \href{https://link.springer.com/book/10.1007/978-1-4612-0653-8}{Matrix
  Analysis}, Vol. 169, Springer, 1997.
\newline\urlprefix\url{https://link.springer.com/book/10.1007/978-1-4612-0653-8}

\bibitem{MORe}
inspire AG, \href{https://www.more-simulations.ch/}{{MORe}} (2020).
\newline\urlprefix\url{https://www.more-simulations.ch/}

\end{thebibliography}

\end{document}